%% file: Thesis.tex
\documentclass[12pt]{ociamthesis} 
\input{mypreamble_2018}

\usepackage{amssymb, amsmath}
\usepackage{afterpage}

\title{Determinantal processes and stochastic domination}   

\author{Raghavendra Tripathi}             
\college{Department of Mathematics}  

\renewcommand{\submittedtext}{Thesis Submitted\\
	under the supervision of\\
	\emph{Prof. Manjunath Krishnapur}\\
	 for fulfillment of}
\degree{M.S.}     
\degreedate{2019}         

\begin{document}

\baselineskip=18pt plus1pt

\setcounter{secnumdepth}{3}
\setcounter{tocdepth}{3}

\maketitle                  
\include{declaration}
\include{acknowlegements}   
\include{abstract}          

\begin{romanpages}          
\tableofcontents            
\end{romanpages}            
\include{chapter1}
\include{chapter2}
\include{chapter3}

\include{chapter4}
\include{chapter5}

\addcontentsline{toc}{chapter}{References}
\renewcommand{\bibname}{References}
\bibliography{refs} 
\bibliographystyle{plain}  

\end{document}

%% file: mypreamble_2018.tex
\usepackage{amsfonts}
\usepackage{amssymb}
\usepackage{amsmath}
\usepackage{amsthm}
\usepackage{amsbsy}
\usepackage{amssymb} 
\usepackage{verbatim}
\usepackage{bm} 
\usepackage{paralist} 
 \usepackage{color}
 \usepackage{mathrsfs}
 \usepackage{graphicx}
\usepackage{hyperref}

\newcommand{\E}{\mathbf{E}}
\def\P{\mathbf{P}}

\newcommand{\Poi}{\mbox{\rm Poi}}
\newcommand{\EE}[1]{\E\l[#1\r]}

\newcommand{\X}{\mathcal{X}}

\newcommand{\mat}[4]{\l[\begin{array}{cc}  #1 & #2 \\ #3 & #4  \end{array} \r]}

\def\summ{\sum\limits}
\def\intt{\int\limits}

\def\prodd{\prod\limits}

\def\to{\rightarrow}

\def\sgn{{\mb{sgn}}}
\newcommand\Tr{{\mbox{Tr}}}

\def\mb{\mbox}

\newcommand{\para}[1]{\vspace{6mm}\noindent{\bfseries #1:}}

\newcommand{\parag}[1]{\vspace{5mm}\noindent{\bfseries #1}}

\def\l{\left}
\def\r{\right}
\def\<{\langle}
\def\>{\rangle}

\newcommand{\ba}{\[\begin{aligned}}
\newcommand{\ea}{\end{aligned}\]}
\newcommand\mnote[1]{} 
\newcommand{\beq}[1]{\begin{equation}\label{#1}}
\newcommand\eeq{\end{equation}}
\newcommand\ben{\begin{equation}}
\newcommand\een{\end{equation}}
\newcommand\bes{\begin{eqnarray*}}
\newcommand\ees{\end{eqnarray*}}
\newcommand\besn{\begin{eqnarray}}
\newcommand\eesn{\end{eqnarray}}

\def\bthm{\begin{theorem}}
\def\ethm{\end{theorem}}
\def\bdefn{\begin{definition}}
\def\edefn{\end{definition}}
\newcommand{\benu}{\begin{enumerate}\setlength\itemsep{6pt}}
\newcommand{\beit}{\begin{itemize}\setlength\itemsep{3pt}}
\def\eenu{\end{enumerate}}
\def\eeit{\end{itemize}}
\def\beds{\begin{description}}
\def\eeds{\end{description}}
\def\bepr{\begin{problem}}
\def\eepr{\end{problem}}

\def\bprf{\begin{proof}}
\def\eprf{\end{proof}}
\def\berk{\begin{remark}}
\def\eerk{\end{remark}}
\def\bex{\begin{exercise}}
\def\eex{\end{exercise}}
\def\beg{\begin{example}}
\def\eeg{\end{example}}

\def\AA{{\mathcal A}}
\def\BB{{\mathcal B}}

\def\FF{{\mathcal F}}

\def\TT{{\mathcal T}}

\def\XX{{\mathcal X}}
\def\YY{{\mathcal Y}}

\def\C{\mathbb{C}}
\def\N{\mathbb{N}}
\def\R{\mathbb{R}}

\newcommand{\sm}{{\raise0.3ex\hbox{$\scriptstyle \setminus$}}}

\def\del{\delta}
\def\eps{\epsilon}
\renewcommand\phi{\varphi}

\def\lam{\lambda}

\theoremstyle{plain} 
    \newtheorem{theorem}{Theorem}
    \newtheorem{lemma}[theorem]{Lemma}
    \newtheorem{proposition}[theorem]{Proposition}
    \newtheorem{corollary}[theorem]{Corollary}
    \newtheorem{claim}[theorem]{Claim}

\theoremstyle{definition} 
    \newtheorem{definition}[theorem]{Definition}

    \newtheorem{exercise}[theorem]{Exercise}
    \newtheorem{problem}[theorem]{Problem}
        \newtheorem{remark}[theorem]{Remark}
    \newtheorem{example}[theorem]{Example}

%% file: declaration.tex
\null\newpage
\begin{declaration}
	I hereby declare that the thesis entitled `Determinantal processes and stochastic domination' submitted by me for the award of M.Sc. degree of the Indian Institute of Science did not form the subject matter for any other thesis submitted by me for any degree or diploma.
	\begin{flushright}
		Raghavendra Tripathi\\
		SR No:10-06-00-10-31-16-1-13660
	\end{flushright}
\end{declaration}

%% file: acknowlegements.tex
\null\newpage
\begin{acknowledgements}

Except the mistakes everything else in this work I owe to many people. I am glad to have an opportunity to express my gratitude to those whose help and support I received during this work. 

First of all, I would like to thank my adviser Prof. Manjunath Krishnapur for his insightful comments, discussions and array of questions during this project. It was his course in Random matrix theory which made me interested in probability, and since then he has continuously helped me navigate my way through vast territory of probability theory. He has very kindly and patiently entertained all my questions and doubts. 

I learnt great deal of mathematics during my coursework at IISc, and it served to fill many gaps which I would have hardly been able to do on my own. I thank all my instructors for their wonderful courses and their patience to deal with my doubts while I was embarking on a journey to the beautiful world of mathematics. 

I am also grateful to my fellow students at IISc for many exciting discussions. I am particularly thankful to Abhay Jindal and Shubham Rastogi for being earnest proof readers and pointing out my mistakes-- which I make with very high probability. I thank Poornendu Singh and Mayuresh Londhe for their late night discussions and tea-- both of which has been equally essential to me. I also thank Mayuresh Londhe for directing me to many new frontiers of mathematics, and the discussions from which I learnt a lot of mathematics which I would have otherwise not known. I would also like to thank my friends in other departments at IISc for their great company. In particular, I must thank Debashree Behera for helping me with innumerable things and keeping me concentrated on my work. I must also thank Prakriti canteen--which has been a quasi-permanent place for my mathematical discussions with friends. A major part of this work was written sitting in Prakriti. 

While I learnt rigorous mathematics after coming to IISc, I must thank Dr. Mukund Madhav Mishra at Delhi University for his excellent teaching and inspiring me to pursuing mathematics. Equally important were the courses offered by Prof. S. Bagai, Dr. Umesha Kumar, Dr. Yuthika Gadhyan, Dr. Sulbha Arora and many others.

Undeniably I owe a lot to my teachers at school and college, my friends and my parents for shaping me as a human being. I can not express my gratitude towards them in mere words. My parents always stood by my side while I was making an excursion from physics to mathematics via engineering, and encouraged to pursue the mathematics. It would have been impossible to come this far without their support. I thank my uncle who was my first math teacher after high school. I also thank my elder sister who helped me continue my education at a time when I felt that I won't be able to. Needless to say that I am thankful to my friends who have always kept complaining, (for not returning their calls) but who never quit. I hope that they will be happy to see my thesis--even if they do not understand.

I have received support from numerous other people, and it would be impossible to name everyone here. But, I thank everyone who has helped and supported me in any direct or indirect ways.

\end{acknowledgements}

%% file: abstract.tex
\begin{abstract}
In this thesis we explore the stochastic domination in determinantal processes. Lyons (2003) showed that if $K_1\le K_2$ are two finite rank projection kernels and $P_1, P_2$ are determinantal measures associated with them, then $P_2$ stochastically dominates $P_1$, written $P_1 \prec P_2,$ that is for every increasing event $\mathcal{A}$ we have $P_1(\mathcal{A})\le P_2(\mathcal{A})$.  We give a simpler proof of Lyons’ result which avoids the machinery of exterior algebra used in the original proof of Lyons and also provides a unified approach of proving the result in discrete as well as continuous case.

R. Basu and S. Ganguly (2019) proved the stochastic domination between the largest eigenvalue of Wishart matrix ensemble $W(n, n)$ and $W(n-1, n+1)$ invoking Lyons' theorem. It is well known that the largest eigenvalue of Wishart ensemble $W(m, n)$ has the same distribution as the directed last-passage time $G(m, n)$ on $\mathbb{Z}^2$ with i.i.d. exponential weights. Thus, Basu and Ganguly obtain the stochastic domination between $G(m, n)$ and $G(m-1, n+1).$

It is also known that the largest eigenvalue of the Meixner ensemble $M(m, n)$ has the same distribution as the directed last passage time $G(m, n)$ on $\mathbb{Z}^2$ with i.i.d. geometric weights. We prove another stochastic domination result which combined with the Lyons’ theorem gives the stochastic domination between the largest eigenvalues of Meixner ensemble $M(n, n)$ and $M(n-1, n+1),$ which in turn proves that the directed last passage time (with i.i.d. geometric weights) $G(n, n)$ stochastically dominates $G(n-1, n+1).$
\end{abstract}
\afterpage{\null\newpage}

%% file: chapter1.tex
\null\newpage
\chapter{Point processes}
This chapter aims to provide the background for the upcoming chapters. The primary object of study in this thesis is a determinantal point process and stochatic domination for a special type of determinantal process. Before we specialize to the main theme of the thesis, we will introduce a general point process. There are different possible approaches to introduce the point processes, some of which are specially suitable for specific kind of point processes. The two common approaches to the theory of point process is $a)$ through random sequence of points, and $b)$ through the theory of random measures. In this chapter we briefly describe the two approaches. 

In order to give a complete background for the upcoming chapters we will also describe the notion on stochastic domination and coupling in this chapter.

\section{Definitions and Examples}
 Roughly speaking, a point process is a probability measure on the space of locally finite configurations in some locally compact Polish space. Much of the theory of the point process is inspired from physics and inadvertently a lot of terminology has been borrowed from physics. The points in a configuration are also referred to as particles. Before we give a rigorous definition of a point process, let us look into some simple examples to get an intuition.
\beg
\label{twoplayer}
Let $\X$ be a subset of $\N$ which contains every natural number with probability $p$ independently. $\X$ is a random subset of $\N$. This is an example of a point process.
\eeg
The above example is of course too simplistic but it contains the key idea that a point-process is simply a random subset of some set. Another simple example of a point process is given below.
\beg 
\label{eigenv3}
Consider a $3\times 3$ matrix with each entry is independently distributed according to a Bernoulli $p$ distribution. And let $\X$ be the set of eigenvalues of such a matrix. It is clear that $\X$ is a random subset of $\C$, and is an example of a point process.
\eeg
Note that there are only $2^9$ possible matrices in the above example. Using a computer one can explicitly write down all possible values $\X$ takes, with their exact probabilities. Also note that there is nothing special about $3$, or about the Bernoulli random variables. One can in general start with any random matrix ensemble and the set of eigenvalues will give a point process on $\C.$ We will talk more about such processes later.

With the above two examples we are now prepared to make a definition for the point process. As we have already remarked a (simple) point process on a set $S$ is a random subset of $S.$ Throughout this chapter, we assume $S$ is a locally compact, complete separable metric space (Polish space) equipped with the Borel $\sigma$-algebra. We start by identifying a random set with a random (Radon) measure on the Borel $\sigma$-algebra of $S$. Note that given a locally finite subset $A$ of $S$, we can associate a measure $\mu_A$ on $S$ defined by $\mu_A=\summ_{a\in A}\delta_a.$ The locally finite assumption on $A$ guarantees that $\mu_A$ is a Radon measure. On the other hand, if we have a Radon measure $\eta$ which only takes non-negative integer values (or possibly infinity), then one can similarly associate it with a locally finite configuration (i.e. a multiset) on $S$. This allows us to see point process as a `random variable' taking values in the space of Radon measures on $S.$ To make this into a formal definition, we shall always take $S$ to be a locally compact Polish space with a reference Radon measure $\mu.$  Denote by $\mathcal{M}(S)$, the collection of Radon measures on the Borel $\sigma$-algebra of $S$ which takes values in $\N \cup \{0, \infty\}.$ Equip the collection $\mathcal{M}(S)$ with the vague topology (the topology which $\mathcal{M}(S)$ inherits as the subspace of $C_0(S)^*$), that is, $\mu_n\to \mu$ in $\mathcal{M}(S)$ if $\int fd\mu_n \to \int fd\mu$ for every $f\in C_0(S).$

It is well known that $\mathcal{M}(S)$ is a complete separable metric space. This identification allows us to define a point process as a random variable on $(S,\mu)$ taking value in $\mathcal{M}(S).$ 

\definition[Point Process]{A point process $\X$ on $(S,\mu)$ is a random finite non-negative integer valued  Radon measure on $S.$ It is called a simple point process if $\X(\{s\})\le 1$ for every $s\in S,$ almost surely.}

It is instructive to think of a simple point process as a random discrete subset of $S.$ It should be pointed out that by the definition of the simple point process, $\X(D)$ is the random variable which counts the number of points (or particles) in the set $D$, for any Borel subset $D\subset S$. The measurability of $\X$ turns out to be equivalent to the measurability of random variables $\X(D)$ for every Borel subset $D\subset S.$

Let us explore a few more examples to understand these point processes better.

\beg[Discrete Poisson process]
Let $S$ be a finite or countable set with a Radon measure $\mu.$ And let $X$ be random multiset of $S$ where the multiplicity of each $x\in \X$ is an independent Poisson with intensity $\mu\{x\}.$ Equivalently $\X$ is random measure defined as $\sum\limits_{x\in S}P_x\delta_{x},$ where $P_x, x\in S,$ are independent random variables and $P_x\sim \mb{Pois}(\mu\{x\}).$
\eeg

The above example also affords us an example of non-simple point process. We do have a continuous analogue of the above process which we record below with a caution that the existence of a process with the properties described below is not at all immediate. We refer the interested reader to \cite{Jones}.

\beg[General Poisson process] 
Let $S$ be a locally compact Polish space with a Radon measure $\mu.$ Let $\X$ be the process such that for any $A\subset S$ of finite measure, the number of points in $\X(A)$ is distributed by Poisson random variable $P_A$ with intensity $\mu(A)\le \infty.$ And for any collection of disjoint subsets $A_1, A_2,\dots, A_k$ of finite measure the collection of random variables $\{P_{A_i}:1\le i\le k\}$ is independent.
\eeg

%
%

We now turn towards the question of describing a point process. Inspired by the general theory of stochastic processes, one would imagine that the natural way to describe a point process would be by describing the probabilities of its cylinder sets i.e. by specifying the $\Pr[\X(B_i)=k_i, 1\le i\le m]$ for all $m\ge 1$ and Borel subsets $B_i\subset S$. Of course, in order to define a point process the assignment of probabilities to the cylinder sets must be consistent meaning that 
$$\sum\limits_{0\le k_{m+1}\le \infty}\Pr[\X(B_i)=k_i, 1\le i\le m+1]=\Pr[\X(B_i)=k_i, 1\le i\le m].$$

This indeed is useful and very much in the spirit of general theory of stochastic processes. But this is not the most preferred or the most amenable way to describe a point process. The distribution of a point process is most often described by its \emph{joint intensities/correlation functions}. Of course, there are other ways to describe a point process but we will not get into details here. We also caution the reader the joint intensities do not always exist and even when they do, they need not completely determine a point process, but for all our purposes specifying the joint intensities would be enough. For a short but beautiful discussion of joint intensities we suggest the reader to look into Chapter 1 of \cite{GAF}, and also the survey paper \cite{Peres}, which contains everything necessary for our purposes. For a full treatment of theory of point-process and understanding full nuances, we also refer the reader to \cite{Jones}. Here we content ourselves with the definitions and facts that would be useful to us later. Recall that $(S, \mu)$ is a locally compact Polish space equipped with the Borel $\sigma$-algebra and $\mu$ is a Radon measure on $S$.

\definition[Joint Intensity]{Let $\X$ be a simple point process on $(S, \mu).$ A symmetric, non-negative, locally integrable function $\rho_k:S^k\to \R$ is $k$-th joint-intensity (or correlation function) of $X$ if for any family of mutually disjoint Borel subsets $D_1, \ldots, D_k$
	$$\int_{\prod\limits_{i=1}^{k} D_i} \rho_k(x_1,\dots, x_k)d\mu(x_1)\dots d\mu(x_k)=\EE{\prod\limits_{i=1}^{k}\X(D_i)}.$$ 
}

It is clear that if the joint intensities exist, they are determined uniquely (up to almost everywhere equivalence). The key object of study in this thesis is a class of processes called \emph{determinantal processes} for which the existence of correlation functions/ joint intensities is forced by the definition. Therefore, we will not spend much time on the joint intensities here.

For the sake of completeness we remark that the joint intensities determine the law of the point process if for every compact set $D\subseteq S,$ the probabilities $$\Pr[\X(D)\ge k]\le \exp(-ck),\hspace{3mm} k\ge 1$$ for some positive constant $c.$ The proof of this fact is simple and follows from the fact that under above conditions, the random vector $(\X(D_1,)\ldots, \X(D_k))$ has convergent Laplace transform in a neighborhood of origin for any compact set $D_1,\ldots, D_k$. This allows one to uniquely specify the finite dimensional distributions of the process. Those who are not satisfied with this intuition and insist upon a detailed proof are referred to the chapter 1 of \cite{GAF}. We find it appropriate to mention that the joint-intensities of a point process can be thought of as the counterpart of the moments (more precisely, of factorial moments) of a random variable. It is not hard to see that $$\E\left({\X(D)\choose k}k!\right)=\intt_{D^k}\rho_k(x_1,\ldots,x_k)\prodd_{i\le k}d\mu(x_i).$$ The classical moment problem concerns the question of determining random variable with first $n$-moments specified. The similar questions have been asked in the context of point process by specifying the first few joint-intensities. This does not concern us at this moment, but the beauty of this subject rightfully demands its mention and we refer the reader to \cite{KLS11} for the details.

We end this section by pointing out that for a point-process with fixed deterministic total number of points, say $n$, all the joint intensities $\rho_k$ become identically $0$ for $k>n.$ Another thing which happens is that one can determine the lower order joint-intensities from $\rho_n.$ More precisely we have that 
$$\rho_k(x_1,\ldots, x_k)=\frac{1}{(n-k)!}\intt_{S^{n-k}}\rho_n(x_1,\ldots,x_n)\prodd_{i>k}d\mu(x_i).$$  

To see that it is something worth mentioning, consider the following very simple example of two point processes on a finite set $S=\{1,2,3\}.$ The first process, say $X_1$, is obtained by choosing each element from $S$ independently with probability $\frac{1}{2}.$ Note that the highest order correlation function $\P(1,2,3\in X_1)=\frac{1}{8}$, while $\rho_2(x,y)=\P(x,y\in X_1)=\frac{1}{4}$ for any $x\neq y.$ Now, consider another process $X_2$ on the same set $S$ defined by the following law. Let $1\in X_2.$ And choose $2$ with probability $\frac{1}{4}$ while $3$ with probability $\frac{1}{2}$ independently. Once again $\P(1,2,3\in X_2)=\frac{1}{8},$ but $\P(1,3\in X_2)=\frac{1}{2}, \P(1,2\in X_2)=\frac{1}{4}$ and $\P(2,3\in X_2)=\frac{1}{8}.$ This simple example illustrates that the lower order correlation functions are not always determined by the top-order correlation functions.  

\section{Stochastic domination and coupling}
In this subsection we will introduce the notion of stochastic domination and coupling. Thanks to a theorem due to Strassen\cite{Strassen} these two notions are very intimately related . 

Let us start with some motivation. Consider a sequence of random variables $X_i$ and define $M_n$ to be the maximum of $\{X_i: 1\le i\le n\}.$ It is clear that $M_n\le M_{n+1},$ and this inequality can be interpreted in strongest possible sense. Meaning, if we compare the two random variables $M_n$ and $M_{n+1}$ for each `sample', we will see that $M_n(\omega)\le M_{n+1}(\omega).$ A similar example would be obtained if we consider $S_n:= \sum\limits_{1\le i\le n}Y_i$ where $Y_i$ are all non-negative random variables. We observe that $S_n\le S_{n+1},$ and once again the inequality holds true for each $\omega.$ Let us now look at another example which is slightly more illuminating. 

\begin{example}
	Let $X = X_{\lam}$ and $Y = Y_{\mu}$ be two Poisson random variable with rate $\lam$ and $\mu$, respectively. Suppose $\lam \le \mu.$ Very naively, one might want to think that $X \le Y$ in some suitable sense. Here, we can not say that $X(\omega)\le Y(\omega)$ for each $\omega.$ But, intuitively we know that $Y$ is likely to be bigger than $X.$ This intuition can be translated into rigorous mathematics by noticing that for every real $x,$ $$\mathbb{P}(X\ge x)\le \mathbb{P}(Y\ge x).$$
	Although one can compute the above two probabilities explicitly and show that the above inequality is indeed true, here we give an alternate proof which also serves a greater goal.
	
	We first recall that sum of two independent Poisson random variables $P_1$ and $P_2$ with rate $\mu_1, \mu_2$ respectively, is again a Poisson random variable with rate $\mu_1+\mu_2.$ Therefore, we define (on some probability space) a Poisson random variable $X'\stackrel{d}{=} X_{\lam}$ and a Poisson random variable $Z,$ which is independent of $X'$ and has rate $\mu-\lam.$ By our previous remark $Y_{\mu}=^{d} X'+Z.$ We can immediately see that on this new probability space $X' \le X'+Z$ (almost surely), and therefore
	$$\mathbb{P}(X\ge x)=\mathbb{P}(X'\ge x)\le \mathbb{P}(X'+Z\ge x)=\mathbb{P}(Y\ge x).$$ 
\end{example}

We pause to iterate that we constructed two random variables $Y\stackrel{d}{=} Y':=X'+Z$ and $X\stackrel{d}{=} X',$ on some probability space such that $X'\le Y'$ almost surely. This is an instance of \emph{coupling}, that is a realization of $(X', Y')$ on same probability space such that their marginals agree with the distribution of $X$ and $Y$. With a little thought, one may find it natural to say that $Y$ stochastically dominates $X$ if we can construct a coupling as in the previous example. To restore one's faith in the justice, this turns out to be an equivalent way of defining the stochastic domination and is a well-known result due to Strassen\cite{Strassen}, which we have included as Theorem \ref{StraTheo} for the sake of completeness. 
%
%
%

\definition[Increasing set]{Let $(\Omega, \le)$ be a partially ordered set (with the partial order $\le$). A subset $\AA\subseteq \Omega$ is said to be increasing if $\omega_1\in \AA$ whenever $ \omega_0\le \omega_1$ for some $\omega_0\in\AA.$}

\definition[Stochastic domination for probability measures]{Let $(\Omega, \FF, \le)$ be a partially ordered measurable space (that is $\Omega$ is a partially ordered set equipped with a sigma algebra). Let $\P_1$ and $\P_2$ be two probability measures on $(\Omega, \FF, \le).$ We say that $\P_1$ is stochastically dominated by $\P_2$ (with respect to partial order $\le$), denoted as $\P_1\prec \P_2$, if $\P_1(\AA)\le \P_2(\AA)$ for every increasing subset $\AA\in \FF.$}

It is important to note that the whether a subset $\AA\subseteq \Omega$ is increasing or not depends very much on the partial order on the set $\Omega,$ and as a consequence an statement like $\P_1\prec \P_2$ is meaningful only when the partial order on the underlying space $\Omega$ is fixed. But whenever the partial order in question would be clear from the context, we will just write $\P_1\prec \P_2$ without any mention of the partial order. We also note that an increasing subset $\AA\subseteq \Omega$ need not be measurable, but the definition above asks for $\P_1(\AA)\le \P_2(\AA)$ only for those increasing subsets which are measurable. One may constrict examples most increasing subsets are not measurable, but often the partial order on $\Omega$ is compatible with the $\sigma$-algebra and hence we do not impose any further conditions on the partial order.

As we remarked in the beginning, the notion of Stochastic domination is intimately related to the idea of coupling. Before we end this section, we record a theorem of Strassen which connects coupling with the Stochastic domination. The traditional wisdom regarding coupling is `to have the same source of randomness' for two random variables, which allows one to compare them.

\definition[Coupling]{Let $X$ and $Y$ be two random variables on $(\Omega_1, \FF_1, \P_1)$ and $(\Omega_2, \FF_2, \P_2)$ respectively. A coupling of $X$ and $Y$ is a random vector $(X',Y')$ on a new probability space $(\Omega, \FF, \P)$ such that $X'=^d X$ and $Y'=^d Y.$}

\begin{theorem}[Strassen, 1965]
	Let $(\Omega, \le)$ be a partially ordered finite set with two probability measures, $\mu_1$ and $\mu_2.$ The following are equivalent:
	
	$\bullet$ There is a probability measure $\nu$ on $\{(x,y)\in \Omega\times\Omega: x\le y\}$ whose coordinate projections are $\mu_i.$

	$\bullet$ For each increasing subset $\AA\subseteq \Omega,$ we have $\mu_1(\AA) \le \mu_2(\AA).$
	\label{StraTheo}
\end{theorem} 

The first statement in the theorem is essentially the existence of a coupling i.e. existence of a measure on the product space with the correct marginals, while the second statement is of course saying that $\mu_1$ is stochastically smaller than $\mu_2.$ Observe that under the measure $\nu,$ almost surely, the first component is smaller than the second, which is analogous to the construction we did in the case of Poisson random variables.  

An elegant proof of the above theorem using `min-cut max-flow theorem' can be found in Chapter 10 (Theorem 10.4) \cite{LyonsPeres}. In the remaining of the thesis we will not be concerned with any explicit coupling.

%% file: chapter2.tex
\chapter{Determinantal processes}

In this chapter we introduce the notion of the determinantal point processes. We also record some key properties of these processes which shall be useful later. In order to facilitate the understanding of determinantal processes, we start with discrete case and study the example of the Uniform spanning tree. We will also record some interesting examples of determinantal processes in the continuous setting. 

\section{Definition and properties}
As we have already noted that a point process $\X$ is a random discrete subset of a locally compact Polish space. We now turn towards a special class of point processes which has made its appearance in many different areas of probability, namely the determinantal processes. The systematic study of the determinantal processes began with Macchi's work (1975) on `fermionic processes', although the use of determinantal processes in random matrix theory was known since early 60s. One crucial feature of `fermionic' particles is that they repel each other and determinantal processes capture this interaction. Before we begin the discussion of determinantal processes we remind the reader that throughout this chapter $(S,\mu)$ will be a locally compact Polish space.

\definition{A point process $\X$ on $(S,\mu)$ is said to be \emph{determinantal} if it is simple and there exist a locally integrable function $K:S\times S\to \C$ such that 
	$$\rho_k(x_1,\ldots, x_n)=\det(K(x_i,x_j))_{1\le i,j\le k}$$
	for every $k\ge 1.$}

Determinantal processes satisfy many algebraic identities and that is probably one reason why these processes are so ubiquitous.

We recall that for a general point process the existence of correlation functions is not guaranteed. For a determinantal process the existence of correlation functions is a part of the definition. One may imagine that there would be other processes with similar definitions in which the correlation functions are given by some other algebraic quantities instead of determinant viz. permanent, immanant or pfaffian etc. We wish to point out that such processes have been indeed defined and have been studied. We will not pursue the subject here, but we refer the interested reader to \cite{SoshnikovPfaf}, \cite{SoshnikovDet}, \cite{Peres}, \cite{GAF} for the definitions and examples of such processes which has been of interest.

Coming back to the determinantal processes, we notice that the kernel $K$ cannot be completely arbitrary. For example, as the joint intensities are non-negative and locally integrable it follows that $\det(K(x_i,x_j))_{1\le i,j\le k}$ must be non-negative and locally integrable w.r.t. $\mu^{\otimes k}$. There are other caveats in the definition which one should be careful about. For example, the first correlation function of a determinantal process is given by $\rho_1(x)=K(x,x).$ But as a general measurable function is defined only upto almost everywhere equivalence, the function $K(x,x)$ might not even be well-defined (if $\mu$ is non-atomic the diagonal has measure zero). Of course there are similar issues with higher correlation functions as well. Moreover, the existence and uniqueness of a determinantal process is not immediately obvious from the definition above. 

It is not hard to see that one can modify the measure and Kernel of a determinantal process together without changing the process. For example, consider a determinantal process on $(S,\mu)$ with kernel $K.$ Let $f:S\to \C$ be a function such that $\frac{1}{f}$ is locally square integrable. Define a new measure $d\mu_f=\frac{1}{|f|^2}d\mu$ and kernel $K_f(x,y)=f(x)K(x,y)\overline{f(y)}.$ Then, the same determinantal process can be treated as a determinantal process on $(S,\mu_f)$ with the kernel $K_f$. This shows that there is at least a limited amount of freedom available to us in choosing the measure and kernel pair. In fact, we will exploit this freedom later when we would compare two determinantal processes.

In the upcoming sections we will see some examples of determinantal processes in discrete as well as continuous case. In discrete case -- that is when $S$ is an at most countable set with some random measure (for example counting measure) -- the issue of well-definedness of the correlation function does not arise. Similarly, in the general case if the kernel $K(x,y)$ is continuous, the problem is resolved. The examples which we will be dealing with will be of this nature. Therefore, we will not worry about this issue. Yet for the sake of completeness, we must add that the continuity of $K$ is indeed very restrictive and is not required for $K(x, x)$ to be well-defined.

Recall that a kernel $K$ is square-integrable on $S^2,$ if
$$\int\limits_{S^2}|K(x,y)|^2 d\mu(x) d\mu(y)<\infty.$$
Such a kernel $K$ defines an integral operator $\mathcal{K}$ on $L^2(S,\mu).$  Moreover, the operator $\mathcal{K}$ is a Hilbert-Schmidt operator, in particular, it is compact. If additionally we assume that $K(x,y)=\overline{K(y,x)},$ then the integral operator defined by $K$ is also self-adjoint. From the spectral theorem for compact self-adjoint operators, we have that there are at most countably many distinct eigenvalues of $\mathcal{K}$ and all the eigenvalues (except possibly $0$), have finite multiplicities. Moreover, $L^2(S, \mu)$ admits an orthonormal basis of eigenfunctions $\{\phi_i\}$ of $\mathcal{K}$ and we have the following representation for the kernel $K$, $$K(x,y)\stackrel{L^2}{=} \sum\limits_{i=1}^{\infty}\lambda_i\phi_i(x)\overline{\phi_i(y)}.$$
However, the above equality holds only in $L^2,$ and therefore, $K(x, x)$ is still not well-defined. Therefore, we make an extra assumption that the integral operator $\mathcal{K}$ associated with the kernal $K$ is trace class, that is, $\sum\limits_{i}^{\infty}|\lambda_i|<\infty.$ With the assumption that $\mathcal{K}$ is trace-class, we can write $K(x,y)= \sum\limits_{i=1}^{\infty}\lambda_i\phi_i(x)\overline{\phi_i(y)},$  where the sum in the right hand side converges absolutely almost everywhere, that is, there exists $S_1\subseteq S$ such that $\mu(S\setminus S_1)=0$ and the series $K(x,y)= \sum\limits_{i=1}^{\infty}\lambda_i\phi_i(x)\overline{\phi_i}(y)$ converges absolutely on $S_1\times S_1.$ (Of course, in addition it still converges in $L^2.$) This allows us to defined the joint intensities $\rho_k$ on $S^k$ a.e. with respect to $\mu^{\otimes k}$ when $K$ defines a trace class operator. Recall that $K$ is locally square-integrable on $S^2,$ if  $$\int\limits_{D^2}|K(x,y)|^2 d\mu(x) d\mu(y)<\infty$$ for every compact set $D\subseteq S.$ If $K$ is locally square-integrable and Hermitian, then it defines a self-adjoint operator $\mathcal{K}$ on the space of all functions $f\in L^2(S, \mu)$ which vanish $\mu$ a.e. outside some compact subset of $S$. The restriction of $\mathcal{K}$ to $L^2(D, \mu)$ for any compact subset $D\subseteq S,$ say $\mathcal{K}_D,$ is then a compact self-adjoint operator. We say that operator $\mathcal{K}$ is locally trace-class if $\mathcal{K}_D$ is trace-class for every compact subset $D.$ The condition that $\mathcal{K}$ is trace-class is too restrictive, but it suffices to consider the locally square-integrable kernel $K$ such that associated integral operator $\mathcal{K}$ is locally trace class. This turns out to be sufficient for defining the joint intensities $\rho_k$ on $S^k$ a.e. with respect to $\mu^{\otimes k}.$ For the detailed proofs of the above claim we refer to the Chapter 4 of \cite{GAF}. 

Before we proceed further, we must point that generally the kernel $K$ need not be Hermitian, and there are known examples of determinantal processes with non-Hermitian kernels which we shall not pursue here.  Recall from the Chapter 1 that specifying the joint intensities determines the law of a point-process $\X$ only if for every compact set $\X(D)$ has exponentially decaying tail i.e. $\P(\X(D)>k)\le C_De^{-c_Dk}.$ For a determinantal process it is indeed the case and therefore the kernel $K$ of a determinantal process $\X$ specifies the law of $\X$ uniquely. 

\begin{lemma}[Lemma 4.2.6, \cite{GAF}]
	Let $\X$ be a determinantal process with the (hermitian) kernel $K$. Then for any compacts set $D\subseteq S,$ there exists constants $C_D>0, c_d>0$ such that $$\P(\X(D)>k)\le C_De^{-c_Dk}.$$
\end{lemma}
\begin{proof}
	First note that for any compact set $D\subseteq S$ we must have 
	\begin{align*}
	\E\left({\X(D) \choose k}k!\right)&=\int\limits_{D^k}\det(K(x_i,x_j))_{1\le i,j\le k}\prod_{i=1}^{k}d\mu(x_i)\\
	&\le \intt_{D^k}\prod_{i=1}^{k} K(x_i,x_i)\prod_{i=1}^{k}d\mu(x_i)\\
	&=\left(\intt_D K(x,x)d\mu(x)\right)^k  <  \infty.
	\end{align*}
	where the inequality uses Hadamard's inequality for the determinant of positive semi-definite matrices ($\det(M)\le \prod_i (M)_{i,i}$). The finiteness of the last integral follows from the fact that $D$ is compact (recall that the joint intensities are locally integrable). Now for any $s>0,$ we have 
	\begin{align*}
		\E\left((1+s)^{\X(D)}\right) &= \sum\limits_{k\ge 0}s^k\E\left({\X(D)\choose k}\right)\\
		&\le \sum\limits_{k\ge 0}\frac{s^km_D^k}{k!},\hspace{5mm}\text{where}\hspace{2mm} m_D=\int_D K(x,x)d\mu(x)\\
	&=e^{-sm_D}.
	\end{align*}
	Apply Chebyshev's inequality to get 
	$$\P(\X(D)>k)\le (1+s)^{-k}\E\left((1+s)^{\X(D)}\right)\le (1+s)^{-k}e^{-sm_D}$$
	which proves the claim.
\end{proof}
In the light of this lemma and the discussion in the chapter 1, it follows that the determinantal processes are uniquely determined by their (Hermitian) kernels. We must also caution that not all kernels $K$, even when $K$ is Hermitian, determine a determinantal process. The following theorem gives a simple criterion for determining which Hermitian kernels determine a determinantal process. 

\begin{theorem}[Macchi, Soshnikov]
	Let $K$ be a Hermitian kernel on $(S,\mu)$ which defines a locally-trace class operator $\mathcal{K}$ on $L^2(S,\mu).$ Then $K$ determines a determinantal process if and only if $0\le \mathcal{K}\le I.$
	\label{MacchiSo}
\end{theorem}

We omit the proof of the theorem but we refer the reader to \cite{Soshnikov} for the original proof of Soshnikov. An alternate proof of the theorem can be found in the survey article \cite{Peres}. We also wish to point out that there are no analogous results known for the necessary and sufficient conditions for a kernel to determine a determine a determinantal process when $\mathcal{K}$ is not Hermitian. 

A particular case of the above theorem (although, it is used to prove the above theorem in \cite{GAF}) is obtained when the operator $\mathcal{K}$ is a finite rank projection. The examples we would be working with will usually be of this nature, therefore we record it as a lemma. 

\begin{lemma}
Suppose $\{\phi_i\}_{i=1}^{n}$ is an orthonormal set in $L^2(S,\mu).$ Then there exists a determinantal process with the kernel $K(x,y)=\sum\limits_{i=1}^{n}\phi_i(x)\overline{\phi_{i}(y)}.$	
\label{finiteprojectiondpp}
\end{lemma}
We give a proof of this lemma which is taken from \cite{GAF}, because it contains some elementary but useful ideas. An important property of the determinantal process obtained from the finite rank projection kernel of rank say $n$, is that such a process almost surely contains $n$ points. That is this determinantal process has fixed, finite, deterministic number of total points. The proof is not hard. It is clear that the matrix $K(x_i,x_j)_{1\le i,j\le m}$ has rank at most $n.$ Therefore, $\E\left({\X(S)\choose k}\right)=0$ for every $k\ge n+1,$ which means $|\X(S)|\le n$ almost surely.   But, the first intensity $\rho_1(x)=K(x,x),$ which means
\begin{align*}
\E\left(X(S)\right)&=\intt_S K(x,x)d\mu(x)\\
&=\summ_{i=1}^{n}\int_S|\phi_i(x)|^2d\mu(x)\\
&= n.
\end{align*}
It is clear from the above discussion that $X(S)=n$ almost surely ($X(S)$ is a random variable bounded almost surely by $n,$ but has expectation $n$). We recall from chapter 1, that for such a process the lower order intensity functions are determined by $\rho_n.$ This fact will be useful in the proof the lemma \ref{finiteprojectiondpp}.
  
\begin{proof}[Proof of lemma \ref{finiteprojectiondpp}]

First observe that for any $x_1,\ldots, x_n$, we have that \[(K(x_i,x_j))_{1\le i,j\le n}=AA^*,\] where $A(i,k)=\phi_k(x_i)$, that is, $K$ is positive semi-definite. It, therefore, follows that $\det(K(x_i,x_j))_{1\le i,j\le k}\ge 0$ for any $k.$ A straightforward computation, using the fact that $\{\phi_i\}$ is orthonormal, one can show that 
	\[\intt_{S^n}\det(K(x_i,x_j))_{1\le i,j\le n}=n!.\]
It therefore, follows that $\frac{1}{n!}\det(k(x_i,x_j))_{1\le i,j\le n}$ is a probability density on $S^n.$ Treating the random variable thus obtained as unlabeled points in $S$, we get the joint intensity $\rho_n(x_1,\ldots, x_n)=\det(K(x_i,x_j))_{1\le i,j\le n}.$ As we remarked earlier, this determines the lower order joint intesities via the formula
$$\rho_k(x_1,\ldots,x_k)=\frac{1}{(n-k)!}\intt_{S^{n-k}}\rho_n(x_1,\ldots,x_n)\prodd_{i>k}d\mu(x_i).$$ 
Following \cite{GAF}, we compute $\rho_{n-1}$ below, and leave the details to obtain lower order intensity functions.
	\begin{align*}
	\rho_{n-1}(x_1,\ldots,x_{n-1})&=\intt_S \rho_n(x_1,\ldots,x_n)d\mu(x_n)\\
	&=\intt_S \det(K(x_i,x_j))_{1\le i,j\le n}d\mu(x_n)
	\end{align*}
	which can be expanded into
	$$\summ_{\pi, \sigma \in S_n}\sgn(\pi\sigma)\prodd_{i=1}^{n-1}\phi_{\pi(k)}(x_k)\overline{\phi}_{\pi(k)}(x_k)\intt_S\phi_{\pi(n)}(x_n)\phi_{\sigma(n)}(x_n)d\mu(x_n).$$
	Using the fact that $\phi_i$ were orthonormal, we se that the integral in the above expression is non-zero only when $\pi(n)=\sigma(n),$ therefore it is equal to
	$$\summ_{j=1}^{n-1}\summ_{\stackrel{\pi, \sigma \in S_n:}{\pi(n)=\sigma(n)=j}}\sgn(\pi\sigma)\prodd_{i=1}^{n-1}\phi_{\pi(k)}(x_k)\overline{\phi}_{\pi(k)}(x_k).
	$$
	Observing that if $\pi$ and $\sigma$ both send $n$ to $j$, we can treat them as a permutation of $\{1,\ldots, n-1\}$ in a natural way, one obtains that
	\begin{align*}
	\summ_{j=1}^{n-1}\summ_{\stackrel{\pi, \sigma \in S_n:}{\pi(n)=\sigma(n)=j}}\sgn(\pi\sigma)&\prodd_{i=1}^{n-1}\phi_{\pi(k)}(x_k)\overline{\phi}_{\pi(k)}(x_k)\\ &= \summ_{j=1}^{n-1} \det(\phi_k(x_i))_{\stackrel{1\le i\le n-1,}{k\neq j}}\det(\overline{\phi}_k(x_i))_{\stackrel{1\le i\le n-1,}{k\neq j}}.
	\end{align*} 
	An application of Cauchy-Binet formula now yield the desired formula for the correlation function.
\end{proof}

\begin{remark} We wish to recall here that a point-process is a random measure. In the above proof we are treating the law of $\X$ as a probability measure on $S^n$. In the next chapter we will be comparing the determinantal processes with the kernels $K_1(x,y)=\summ_{i=1}^n \phi_i(x)\overline{\phi}_i(y)$ and $K_2(x,y)=\summ_{i=1}^{n+1} \phi_i(x)\overline{\phi}_i(y)$ respectively. As we have seen in Chapter 1, that we can compare two measures on some partially ordered set, in order to compare these processes it is useful to keep in mind that their laws are the probability measures on all finite subsets of $S,$ (or probability measures on $\mathcal{M}(S).$)
\end{remark}

It turns out that any determinantal process with a Hermitian, non-negative definite, trace-class kernel $K$ can be seen as a \emph{mixture} (convex combination of measures) of the determinantal processes with projection kernel. And if the eigenvalues of the integral operator associated with the kernel $K$ are $\lambda_k$ (recall that it follows from Theorem \eqref{MacchiSo} that $\lambda_k\le 1$), $k\ge 1$ then the total number of points in the process is distributed according to the sum of independent Bernoulli($\lambda_k$) random variables. Therefore, for most purposes one can restrict one's attention to studying the determinantal processes with finite rank projection kernels. 

Another interesting example of determinantal process is obtained from \emph{bi-orthogonal ensemble}, which can be seen as a generalization of the determinantal processes obtained from finite rank projections. 
\definition[Bi-orthogonal ensemble]{Consider a state space $E$ (locally compact Polish space) with a reference (Radon) measure $\mu$ on it. An $n$-point bi-orthogonal ensemble on $E$ is a measure on $E^n$ given by 
 	\[\P_n(dx_1,\ldots, dx_n):= C_n\det[\phi_i(x_j)]_{i,j=1}^{n}\det[\psi_i(x_j)]_{i,j=1}^{n}\prod\limits_{i=1}^{n}\mu(dx_i),\]
 	for some suitable normalization constant $C_n>0$, and function $\phi_i, \psi_i$ such that all the integrals $G_{ij}:=\int \phi_i(x)\psi_j(x)\mu(dx)$ are finite.}
%
%

A proof of the fact that a bi-orthogonal ensemble is a determinantal process can be found in Lemma 4.2.50 of \cite{AGZ}. We will leave this subject here but we refer the reader to \cite{Lyons}, \cite{GAF}, \cite{Peres} for a detailed discussion of determinantal processes and examples thereof. We refer the reader to \cite{GAF} for more probabilistic intuition behind the determinantal processes and an algorithm to generate a determinantal process.

\section{Continuous case}
The examples of determinantal processes in continuous case are abound. The joint law of eigenvalues of various matrix ensembles turn out to be determinantal with projection kernels. We record some examples of determinantal processes in continuous setting here for the sake of completeness but we refer the reader to \cite{GAF}, \cite{Soshnikov}, \cite{Peres} for details.
\begin{example}[Zeroes of Gaussian analytic functions]
	Let $f(z):=\summ_{n=0}^{\infty}a_nz^n$ where $a_n$ are i.i.d standard complex Gaussian random variables. It is not hard to see that it almost surely defines an analytic functions  on the unit disk. The zero set of this function $f$ was shown to be determinantal by Peres and Virag\cite{GAF}. The kernel of this process (with respect to Lebesgue measure on the disk) is given by the Bergman kernel on unit disk i.e. 
	$$K(z,w)=\frac{1}{\pi(1-z\overline{w})^2}$$
\end{example}

Probably the most important and stimulating example of a continuous determinantal process arises as the joint density of eigenvalues of some random matrix ensemble. We will talk about few such ensembles in coming chapters. Here, we record one such example which is known as Ginibre ensemble. 

\begin{example}
	Let $A$ be an $n\times n$ matrix with i.i.d standard complex Gaussian entries. The eigenvalues of $A$ form a determinantal process on $\C$ with kernel
	$$K(z,w)=\frac{1}{\pi}e^{-\frac{1}{2}(|z|^2+|w|^2)+z\overline{w}}.$$	
\end{example}
There are other random matrix ensembles for which the eigenvalues form a determinantal process. For an interested reader we refer to \cite{Forrester}, \cite{AGZ} for many such examples.
\section{Discrete case}
In this section we will deal with a point process defined on a discrete measure space $(S,\mu).$ One can always keep in mind a subset of $\N$ as a model equipped with some reference measure. We rephrase the definition in this setting, in order to make things more transparent. 

\definition{
Let $S$ be an at most countable set. A simple point process $X$ on $S$ is said to be determinantal with symmetric, positive definite kernel $K:S\times S\to \C$ if for any $k\ge 1$ and $x_1,\ldots,x_k\in S,$ we have 
$$\P(x_1,\dots,x_k\in \X)=det[(K(x_i,x_j))_{1\leq i,j,\leq k}].$$}

Let us recall our example \eqref{twoplayer}. Observe that it is a determinantal process with the kernel $K(x,y)=p\delta_{x=y}.$ 

Conversely, let $S=\{1,2\}$ be a set with two elements. Let $K$ be a symmetric matrix $K=\begin{bmatrix}
a & b\\
b & c
\end{bmatrix}$. Define a determinantal process $\X$ on $E$ by declaring $\P(1\in \X)=a, \P(2\in \X)=c, \P(1,2\in \X)=ac-b^2.$ It is easy to verify using inclusion exclusion principle that it defines a probability measure on all subsets of $S$ provided, of course, $1\geq a,c, ac-b^2\geq 0.$ This last condition is fulfilled if we assume that $K$ is positive semi-definite and $K\le I_2$, that is $I_2-K$ is positive semi-definite. 

We remind our readers that for in the above setting the $\P(x_1,\dots,x_k\in \X)$ is nothing but the $k$-point correlation function $p_k$ of the process $\X$. Therefore, the above definition is a mere translation of the definition given in the previous section. Observe that in discrete setting, it is very easy to compute the probabilities of the form $\P(x_1,\ldots, x_k\in \X).$ It would be nice to obtain a similar formula for, say, $\P(x_1,\ldots, x_k\notin \X).$ Indeed, this can be written entirely in terms of the kernel of the process. The following result gives a way to calculate the probabilities of the events like $\P(x_1,\ldots, x_k\in X, x_{k+1},\ldots, x_m\notin X).$

\begin{proposition}
Let $X$ be a determinantal process on an at most countable set $S$ with the kernel $K:S\times S\to \C.$
$$\P(x_1,\ldots, x_k\in X, x_{k+1},\ldots, x_m\notin X)=\det(\tilde{K}_{k,m}(x_i,x_j))_{1\le i,j\le m},$$
where $\tilde{K}_{k,m}(x_i,x_j)=\left\{\begin{array}{cccccc}

K(x_i,x_j), &  i\le k\\
\delta_{i,j}-K(x_i,x_j), &  i\ge k+1
\end{array}.
\right.$
\label{probability_complement}
\end{proposition}
\begin{proof}
	The proof follows from the induction on $m-k.$ When $m-k=0,$ it is just the definition. Now observe that for $m-k\ge 1,$
	\begin{align*}
	&\P(x_1,\ldots, x_k\in X, x_{k+1},\ldots x_m\notin X)=\\
	&\P(x_1,\ldots, x_k\in X, x_{k+2}, \ldots, x_m\notin X)-\P(x_1,\ldots, x_{k+1}\in X, x_{k+2}, \ldots, x_m\notin X).
	\end{align*}
By induction, we have that 
\begin{equation}
\label{firsteq}
\P(x_1,\ldots, x_{k+1}\in X, x_{k+2}, \ldots, x_m\notin X)=\det\l(\tilde{K}_{k+1,m}(x_i,x_j)_{1\le i, j\le m}\r).
\end{equation}
For the sake of notational simplicty, we will write the above matrix $\tilde{K}_{k+1,m}$ as $K_1.$
And, similarly 
\begin{equation}
\label{secondeq}
\P(x_1,\ldots, x_{k}\in X, x_{k+2}, \ldots, x_m\notin X)= \det\l(\tilde{K}_{k,m-1}(x_i,x_j)_{\stackrel{1\le i, j\le m}{i,j \neq k+1}}\r).
\end{equation}
We now observe that $\det\l(\tilde{K}_{k,m-1}(x_i,x_j)_{\stackrel{1\le i, j\le m}{i,j \neq k+1}}\r)= \det\l(\tilde{L}(x_i,x_j)_{1\le i, j\le m}\r)$ where $L$ is an $m\times m$ matrix, whose $k$th row is $(\delta_{i,k})_{i=1}^m$ and all other rows are same as in $K_1.$ Now, observe that the matrix $L$ and $K_1$ have exactly the entries except in $k$-th row. Using multilinearity of the determinant, therefore, we can write that 
\[\det\l(\tilde{K}(x_i,x_j)_{1\le i, j\le m}\r)+\det\l(\tilde{L}(x_i,x_j)_{1\le i, j\le m}\r) = \det(\tilde{K}_{k,m}(x_i,x_j))_{1\le i, j\le m},\]
which proves the desired claim.
\end{proof}
\begin{corollary}
	Let $X$ be a determinantal process on an at most countable set $S,$ with the kernel $K.$ Then $$\P(x_1,\ldots, x_k\notin X)=\det(I_k-K(x_i,x_j))_{1\le i,j\le k}.$$
\end{corollary}

We will now explore some examples of determinantal processes on discrete state space. Probably the most celebrated and interesting example of determinantal process in discrete setting is Uniform spanning tree on a finite graph. In the following section we will study this example in more detail.

\section{Uniform spanning tree}
Let $G=(V,E)$ be a finite, connected graph. Let $S_G$ be the set of spanning tress of $G.$ Observe that $S_G$ is non-empty finite set (The connectedness of the graph is assumed precisely for this purpose). Uniform measure on $S_G$ gives a point process on the set $E$ of the edges of the graph $G.$ A beautiful result due to Burton and Pemantle \cite{burton} states that this process $\TT$ is determinantal with some kernel $M.$ The Burton-Peamntale theorem gives a electric-network interpretation to the kernel and $M(e,f)$ can be given an electrical interpretation. Of course, there are other ways to interpret this kernel, for example as the hitting time of a symmetric random walk. This subject is vast and beautiful, and a wealth of material on this subject can be found in \cite{LyonsPeres}.   

\begin{theorem}[Burton, Pemantle 1994]
Let $G=(V,E)$ be a finite connected graph. Fix an arbitrary orientation of the edges of $G.$ Let $e_1,e_2...,e_k$ be some collection of edges in the graph $G,$ and let $\TT$ be a spanning tree of $G$ chosen uniformly at random from $S_G$. Then, 
$$\Pr[e_1,\ldots, e_k\in \TT]=\det(M(e_i,e_j)_{1\le i,j\le k}),$$
where $M(e_i,e_j)=$ amount of current flowing through the edge $e_j$ under potential applied on the $e_i$ so that net current in the circuit is $1$ unit.  
\label{bur}
\end{theorem}
We do not include the detailed proof of the above theorem here because it will take us too afar. We refer the reader to Chapter 4 of \cite{LyonsPeres} for a thorough discussion of the result and a proof of the theorem using Wilson's algorithm. The proof of the above theorem exploits the connection between spanning trees, random walks and electrical circuits which is interesting to say the least.

We record below the theorem of Kirchoff on number of spanning trees which is interesting in its own right. But more than that it provides an alternate proof of the Burton-Pemantle theorem.

\begin{theorem}[Kirchoff, 1867]
	Let $G=(V,E)$ be a finite graph. Equip the edges of $G$ with an arbitrary but fixed orientation. The vertex edge-incidence matrix $A_G$ of $G$ is a $V\times E$ matrix given by 
	$$A_G(v,e)=\left\{\begin{array}{ccccc}
	0, & \text{if e is not incidence on v}\\
	+1, & \text{if e starts at v}\\
	-1, & \text{if e ends at v} 
	\end{array}.\right.$$
Let $\tilde{A_G}$ be the matrix obtained by deleting the last row of the matrix $A_G.$ Then $N(G),$ the number of spanning tress of $G$, is given by 
$$N(G)=\det\l(\tilde{A}_{G}\tilde{A}_{G}^t\r).$$ 
\end{theorem}
\begin{proof}
	The proof is actually quite simple. Recall that by Cauchy-Binet identity we have that 
	\ba
	\det\l(\tilde{A}_{G}\tilde{A}_{G}^t\r)&=\summ_{\stackrel{S\subseteq E:}{|S|=n-1}}\det(\tilde{A}_G(S))\det(\tilde{A}_G(S)^t)\\
	&= \summ_{\stackrel{S\subseteq E:}{|S|=n-1}}|\det(\tilde{A}_G(S))|^2 ,
	\ea
	where $\tilde{A}_G(S)$ is the submatrix of $\tilde{A}_G$ obtained by selecting only columns indexed by elements in $S$ (keeping them in the same order as in the original matrix). 
	
We now have to observe that if the edges indexed by $S$ contain a cycle in $G,$ then there exists $\epsilon_e\in \{0,+1,-1\}$ such that $\summ_{e\in S}\epsilon_eC_e=0$ where $C_e$ is the column in matrix $\tilde{A}_G$ indexed by $e.$ It follows therefore that if the subgraph induced by $S$ contains a cycle then, $\det(\tilde{A}_G(S))=0.$ Note that if it does not contain a cycle then it has to be spanning tree. Therefore, suffices to prove that the $\det(\tilde{A}_G(S))=\pm 1,$ if $S$ does not induce any cycle. 
	
This claim can be proved using induction. Let us call the edge $e$ which was connected was to the vertex which has been deleted in $\tilde{A}_G.$ In the column indexed by $e$, there is exactly one non-zero entry which is $\pm 1.$ Expand the determinant along that column to get $\det(\tilde{A}_G(S))=\pm \det(B).$ But observe that $B$ is the edge-incidence matrix of the graph obtained by shrinking both ends of $e$ to one vertex. If $S$ induced a spanning tree on $G,$ then $S\setminus \{e\}$ induces a spanning tree on this reduced graph $G/e.$ Therefore, it follows inductively that $\det(\tilde{A}_G(S)=\pm 1.$ (Of course, the base case when $|S|=1$ is trivial.)

\end{proof}
\remark{
Note that it is hidden in the proof that for a subset $P\subseteq E$ such that $|P|=n-1,$ the $\det(\tilde{A}_G(P)\tilde{A}_G^t(P))=1$ if the edge set $P$ gives a spanning tree of $G$ and $0$ otherwise. The probability measure on $2^E$ given by \[\Pr(P)=\frac{\det(\tilde{A}_G(P)\tilde{A}_G^t(P))}{\det(\tilde{A}_{G}\tilde{A}_G^t)},\] if $|P|=n-1$ and $0$ otherwise, is uniform on $S_G.$ It follows from our previous discussion (on Bi-orthogonal ensemble) that the measure $\Pr$ is determinantal. Therefore, the uniform measure on $S_G$ is determinantal.}

We will now study the uniform spanning tree on $K_n,$ complete graph on $n$-vertices, in some detail. One can use Kirchoff's theorem to see that there are $n^{n-2}$ spanning trees of $K_n$. We will not use this directly and therefore we do not bother ourselves with this computation. We instead use Burton-Pemantle theorem to compute various statistics. In order to apply the Burton-Pemantle theorem, we need to compute the matrix $M$ in the theorem (which is also called transfer-current matrix). For an arbitrary graph computing the transfer current matrix may not be an easy task, but for a complete graph it can be done. 

We first note that if we fix $e,f\in K_n$ and apply battery across $e$ so that the net current from one end of the edge $e$ to the other end is $1$ unit. Then, due to symmetry of the network it is clear that if $I_e$ current passes through the edge $e$ then exactly $I_e/2$ current passes through each of remaining edges which emanate from the same vertex as $e$ and also if there are is an edge which does not meet $e$ then current through that edge must be zero. A simple algebra (and Kirchoff's node law from Physics) therefore tells us that 
$$I_e+ (n-2)I_e/2=1\implies I_e= \frac{2}{n}.$$
Therefore, the matrix $M$ can be defined as $M(e,e)=\frac{2}{n},$ and $M(e,f)=\frac{1}{n}$ if $e \neq f$ but $e$ and $f$ originate from the same vertex, (of course due to symmetry this would imply that if $M(e,f)=\frac{1}{n}$ if $e\neq f$ but $e$ and $f$ end at the same vertex and the sign of the current is reversed if one of them starts at a vertex where the other ends), and $0$ otherwise. We note it below for the record that
\ba
M(e,f)=\left\{
\begin{array}{ccccccccc}
	\frac{2}{n}, & e=f\\
	\frac{1}{n}, & e=\overrightarrow{xy}, f=\overrightarrow{xz}\hspace{1mm}\text{or}\hspace{1mm}\overrightarrow{zy}\\
	\frac{-1}{n}& e=\overrightarrow{xy}, f=\overrightarrow{zx}\hspace{1mm}\text{or}\hspace{1mm}\overrightarrow{yz}\\
	0, &\text{otherwise}
\end{array}.
\right.
\ea
\begin{example}

Let us now compute the probability that the graph distance between two vertices in $\TT$ is $k.$ Choose two vertices of $K_n$ uniformly at random. For the sake of simplicity (of notations) we will call the vertex $v_1$ and $v_2.$ It is evident from the symmetry of $K_n$ that it does not matter which two vertex we choose. To motivate the upcoming computations let us begin with the case $k=1.$ Note that $d_{\TT}(v_1,v_2)=1$ if and only if the edge connecting the two vertices, say $e_{12},$ is contained in $\TT.$ Therefore, \[\Pr\{d_{\TT}(v_1,v_2)=1\}=\Pr\{e_{12}\in \TT\}=\frac{2}{n}.\] Similarly, $d_{\TT}(v_1,v_2)=k$ if and only if there are $(k-1)$ vertices $w_1,\ldots,w_{k-1}$ such that the edges \[e_{v_1w_1},e_{w_iw_{i+1}}, e_{w_{k-1}v_2}\in \TT,\] for all $i=1,2\ldots, k-2.$ Also, observe that once the vertices are chosen, choosing the corresponding edges amounts to fixing a permutation of the choosen $(k-1)$ vertices, and due to the uniqueness of paths between two vertices of a tree, it follows that each permutation gives rise to a distinct event. With all these observations we are left with simple algebra which gives us that 
\begin{equation}
\Pr\{d_{\TT}(v_1,v_2)=k\}={n-2 \choose k-1}(k-1)!\Pr\{e_{v_1w_1},e_{w_iw_{i+1}}, e_{w_{k-1}v_2}\in \TT\}.
\label{dist_k}
\end{equation}

We will make a slight detour to compute $\Pr\{e_{v_1w_1},e_{w_iw_{i+1}}, e_{w_{k-1}v_2}\in \TT\}.$ To this end, we invoke the theorem \ref{bur} to obtain that 
\begin{align*}
\Pr\{e_{v_1w_1},e_{w_iw_{i+1}}, e_{w_{k-1}v_2}\in \TT\}&= \det(M(e_1,\ldots, e_k))\\
&= \det\l(\begin{matrix}
\frac{2}{n} & \frac{-1}{n} & 0 & \ldots & 0\\
\frac{-1}{n} & \frac{2}{n} & \frac{1}{n} & \ldots & 0\\
\vdots & \vdots & \vdots & \vdots & \vdots \\
0 & \ldots & \frac{-1}{n} & \frac{2}{n} & \frac{-1}{n}\\
0 & 0 & \ldots  & \frac{-1}{n} & \frac{2}{n}\\
\end{matrix}\r) \\
&= \frac{1}{n^k}  \det\l(\begin{matrix}
2 & -1 & 0 & \ldots & 0\\
-1 & 2 & -1 & \ldots & 0\\
\vdots & \vdots & \vdots & \vdots & \vdots \\
0 & \ldots & -1 & 2 & -1\\
0 & 0 & \ldots  & -1 & 2\\
\end{matrix}\r)\\
&= \frac{k+1}{n^k}.
\end{align*}
From \eqref{dist_k} and the above computation it follows that 
\begin{equation}
\Pr\{d_{\TT}(v_1,v_2)=k\} = \frac{k+1}{n}\prod\limits_{1\le i\le k-1}\l(1-\frac{i+1}{n}\r).
\label{distance_k_formula}
\end{equation}
\end{example}

The above examples suggest us that we should scale the the edge of the spanning tree by $n^{-\frac{1}{2}}$ as see the limit. Indeed observe that for $k=x\sqrt{n}$ we have that 
$$\prodd_{i=1}^{k}\l(1-\frac{i}{n}\r)\approx e^{-x^2/2}.$$
We will elucidate upon the idea here because, it would be important later. Note that $$\Pr\{d_{\TT}(v_1,v_2)=k\} = \frac{k+1}{n-1}\frac{n-k}{n}\frac{(n)_k}{n^k}.$$
We recall that for $0<t<1,$ we have that $e^{-t/(1-t)}<1-t<e^{-t}.$ And therefore,$$e^{-\frac{k^2}{2(n-k)}}<\frac{(n)_k}{n^k}<e^{-\frac{1}{n}{k\choose 2}}.$$
A tedious but straightforward calculus therefore yields that for \[n^{-\frac{1}{2}-\epsilon}<k<1+n^{\frac{1}{2}+\epsilon}\] we get have
\[\Pr\{d_{\TT}(v_1,v_2)= k\}=\frac{k}{n}e^{-k^2/2n}+O(n^{-1+\epsilon}).\]
Which with a little more involved calculus shows that 
$$\Pr\{\frac{1}{\sqrt{n}}d_{\TT}(v_1,v_2)\le x\}=1-e^{-x^2/2}+o(1)\hspace{3mm}\text{as}\hspace{2mm}n\to \infty.$$
In other words this shows that
$$\frac{1}{\sqrt{n}}d_{\TT}(v_1,v_2)\stackrel{d}{\to}R,$$
where $R$ is the Rayleigh random variable, that is a random variable with density given by $xe^{-x^2/2}$ on $\R_{+}$. 

\begin{example}
We can, as in the above example, choose $k$ vertices from $K_n$ uniformly at random. We are interested in understanding how does the tree spanned by $k$ randomly chosen vertices look like in $\TT$? So let us first fix a shape $\mathbf{t}$ such that $\mathbf{t}$ has exactly $k$-leaves and $2k-2$ vertices and therefore $2k-3$ legs (say $L_1,\ldots, L_{2k-3}$ in some arbitrary but fixed order). We ask for the probability that random chosen vertices $v_1,\ldots, v_k$ span a tree with shape $\mathbf{t}$ and $L_i=m_i$ for $i=1,\ldots, 2k-3.$
	
We will do as we did in the previous example. First write a tree with $k$-leaves with leaves labelled $v_1,\ldots,v_k.$ Now first choose $k-2$ nodes or hubs from $n-k$ vertices and put arrange them in some order, thereafter make the skeleton/shape $\mathbf{t}$ and put $m_i-1$ dots on leg $L_i.$ Choose $\summ_{i=1}^{2k-3}(m_i-1)=m-2k+3$ vertices from remaining and arrange them on dots marked on the legs. This will fix the edges $e_1,\ldots, e_m$ and we need to compute the probability that $\TT$ contains all these edges. 
	
One can inductively keep reducing the length of a leg and finally reduce to a tree with fewer legs, to get a recurrence relation for the determinant. It turns out that the determinant in this case is $\frac{m+1}{n^m}$, where $m=\summ_{i=1}^{2k-3}m_i.$

Combining all this one can get that 
	$$\Pr\{\mathbf{t};L_1=m_1,\ldots, L_{2k-3}=m_{2k-3}\}=\frac{(n-k)!}{(n-m-1)!}\frac{m+1}{n^m}.$$
Note the similarity of this probability with what we obtained in the previous examples. Indeed one can show that if we scale all the edge-lengths by $n^{-\frac{1}{2}}$, this joint distribution converges to the following density 
	$$f(\mathbf{t};x_1,\ldots,x_{2k-3})=\left(\summ_{i=1}^{2k-3}x_i\right)\exp\left(-\frac{1}{2}\left(\summ_{i=1}^{2k-3}x_i\right)^2\right).$$  
\end{example}
We must note that the density $f(\mathbf{t};x_1,\ldots,x_{2k-3})$ obtained above is the finite dimensional distribution of the Brownian continuum random tree (Brownian CRT). In a series of papers (see \cite{CRTI}, \cite{CRTII}, \cite{CRTIII}) Aldous developed a general theory of continuum random trees. A beautiful overview of CRT can be found in \cite{CRTII}. The above results are already contained in $\cite{CRTII}.$ Similar results can also be obtained for other class of random graphs as already shown in $\cite{CRTII}.$ Usually these results are obtained by random walk algorithms (for example Wilson algorithm or Aldous-Broder Algorithm), but here we use the determinantal formulas to obtain the same results.

One can also  analyze the degree of a vertex in uniform spanning tree on $K_n.$ Due to symmetry it does not matter which vertex do we choose. We will fix a vertex and call it $v.$ Note that the degree of a vertex in $\TT$ can not be 0. Once again we will motivate the upcoming computations by doing a simple case first. Let us try to compute the probability that degree of the vertex $v$ is $1$ in $\TT.$ 
\begin{example}

Note that there are $(n-1)$ edges starting at the vertex $v,$ and the degree of $v$ would be equal to $1$ if and only if exactly one of these edges belong to $\TT$ and remaining $(n-2)$ do not. Let us call these edges to be $e_1,\ldots, e_{n-1},$ and compute the $\Pr\{e_1\in \TT, e_2,\ldots, e_{n-1}\notin \TT\}.$ Observe that
	$$\Pr\{\mbox{degree}(v)=1\}= (n-1)\Pr\{e_1\in \TT, e_2,\ldots, e_{n-1}\notin \TT\}.$$
	
In order to compute the required probability, we first note that \[\Pr\{e_1\in \TT, e_2,\ldots, e_{n-1}\notin \TT\}= \Pr\{e_2,\ldots,e_{n-1}\notin\TT \}.\] This follows since we know that there has to be at least edge which connects the vertex $v$ in $\TT$. This will help us simplify some computations. We invoke the corollary to the theorem \ref{probability_complement} to compute the $\Pr\{e_2,\ldots,e_{n-1}\notin\TT \}.$
	\begin{align*}
	\Pr\{e_1\in \TT, e_2,\ldots, e_{n-1}\notin \TT\}&=\Pr\{e_2,\ldots,e_{n-1}\notin\TT \}\\
	&= \det\l( 
	\begin{matrix}
	\frac{n-2}{n} & \frac{-1}{n} & \ldots & \frac{-1}{n}\\
	\frac{-1}{n} & \frac{n-2}{n} & \ldots &  \frac{-1}{n}\\
	\vdots & \vdots & \vdots & \vdots \\
	\frac{-1}{n} & \frac{-1}{n} & \ldots &  \frac{n-2}{n}
	\end{matrix}
	\r)\\
	&= \frac{1}{n^{n-2}}(n-1)^{n-3}.
	\end{align*}
As we have already observed that \[\Pr\{\mbox{degree}(v)=1\}= (n-1)\Pr\{e_1\in \TT, e_2,\ldots, e_{n-1}\notin \TT\},\] it follows that 
	\[\Pr\{\mbox{degree}(v)=1\} = (n+1)\frac{1}{n}\l(1-\frac{1}{n}\r)^{n-3}\to e^{-1}\hspace{1mm}\text{as}\hspace{1mm} n\to \infty.\]
\end{example}

More generally one can show that $\mbox{degree}(v)\to 1+\Poi(1).$ To this end, let us fix a vertex $v$ and edges as in the previous example and denote by $I_j$ the indicator function of the edge $e_j$ in $\TT.$ Note that 
\begin{align*}
\EE{\mbox{degree}(v)_{(k)}}&:= \EE{\mbox{degree}(v)(\mbox{degree}(v-1))\ldots (\mbox{degree}(v)-k+1)}\\
&= \sum\limits_{i_1,...,i_k \text{distinct}}\EE{\prod\limits_{1\le j\le k}I_{i_j}}.
\end{align*}
The key thing to note here is that $\EE{\prod\limits_{1\le j\le k}I_{i_j}} = \Pr\{e_{i_1}, \ldots, e_{i_k}\in \TT\}$. And, thanks to theorem \ref{bur} computing this probability is very straightforward for the complete graphs. 

\begin{example}
In this example we continue the discussion in the previous paragraph and calculate the $\EE{\mbox{degree}(v)_{(k)}}.$ We first recall that 
$$\Pr\{e_{i_1}, \ldots, e_{i_k}\in \TT\}= \frac{k+1}{n^k}.$$
Observe that this probability is independent of the precise $k$-tuple chosen to compute the probability. And, therefore all we need to do is to multiply it by all possible $k$-tuples of edges chosen from the total of $(n-1)$ edges incident at the vertex $v.$ In the light of above discussion therefore we obtain
\begin{align*}
\EE{\mbox{degree}(v)_{(k)}} &= {n-1 \choose k}k!\frac{k+1}{n^k}\\
&= (n-1)(n-2)\ldots (n-k)\frac{k+1}{n^k}\\
&= (k+1)\prod_{i=1}^{k}\l(1-\frac{i}{n}\r).
\end{align*}
Observe that $\EE{\mbox{degree}(v)_{(k)}}\to (k+1)$ as $n\to \infty.$ Observe that it is also the factorial moment for $(1+\mbox{Poisson}(1))$ random variable. It follows that $\mbox{degree}(v)\to 1+\Poi(1)$ in distribution.
\end{example}
Recall that a vertex with degree $1$ is called a leaf. In the above example we have computed that $\Pr(\mbox{degree}(v)=1)=(1-\frac{1}{n})^{n-2}.$ With this we can try and estimate the number of leaves in $\TT.$ 
\beg 
Let $I_v$ denote the indicator function of the event that the vertex $v$ is a leaf. Clearly, $\EE{\text{no. of leaves in }\TT}=\summ_{v\in K_n}\EE{I_v}=n(1-\frac{1}{n})^{n-2}.$ Evidently, we obtain 
\[\EE{\frac{\text{no. of leaves}}{n}}\to e^{-1}.\]
That is a positive fraction of the vertices are leaves. In fact, we can do better by observing that 
\[\mbox{Var}\l(\frac{1}{n}\summ_{v\in K_n}I_v\r)=\frac{1}{n}\mbox{Var}(I_v)+\frac{n-1}{n}\mbox{Cov}(I_u,I_v).\]
$I_v$ is a Bernoulli random variable therefore the $\mbox{Var}(I_v)=\Pr(\text{v is a leaf})(1-\Pr(\text{v is a leaf}))\to (1-e^{-1})e^{-1}.$ Also note that $$\mbox{Cov}(I_v,I_u)=\l(1-\frac{2}{n}\r)^{n-2}-\l(1-\frac{1}{n}\r)^{2(n-2)}.$$
It therefore follows that $\mbox{Var}\l(\frac{1}{n}\summ_{v\in K_n}I_v\r)\to 0$ as $n\to \infty.$ Applying Markov's inequality we get that $$\frac{\text{no. of leaves}}{n}\stackrel{P}{\to} e^{-1}.$$
\eeg

%% file: chapter3.tex
\chapter{Stochastic domination}
As already remarked in the previous chapters, determinantal processes exhibit some stochastic domination. In this chapter we shall explore some results on stochastic domination in such processes and also see some applications.  

\section{Stochastic domination for finite rank projections}
In order to make this chapter largely self contained, we will recall some basic notions already introduced in previous chapter. Let $(E,\FF,\mu)$ be a measure space and let $K(x,y)=\sum_{k=1}^n\phi_k(x)\bar{\phi_k}(y)$ where $\{\phi_1,\ldots ,\phi_n\}$ is an orthonormal set in $L^2(E,\mu)$. Let $(X_1,\ldots ,X_n)$ is a random tuple in $E^n$ having density $f(x_1,\ldots ,x_n)=\frac{1}{n!}\det(K(x_i,x_j))_{i,j\le n}$ with respect to $\mu^{\otimes n}$. The point process (meaning, a random integer-valued measure) $\XX :=\del_{X_1}+\ldots +\del_{X_n}$ is a determinantal point process with kernel $K$ (w.r.t. the measure $\mu$).

Our goal in this chapter is to compare two such processes whose kernels are given by $K_1(x,y):=\summ_{i=1}^{n}\phi_i(x)\overline{\phi}_i(y)$ and $K_2(x,y):=\summ_{i=1}^{n+1}\phi_i(x)\overline{\phi}_i(y)$ respectively. Evidently, the law of these processes, say $\P_1$ and $\P_2$ respectively, are probability measures on $\mathcal{M}(E)$ the space of non-negative integer valued Radon measures on $E$. The space $\mathcal{M}(E)$ is a partially ordered set and being a locally compact Polish space it is also equipped with a natural Borel sigma-algebra. As we have already defined in the first chapter, a measurable subset $\AA$ of $\mathcal{M}(E)$ is said to be increasing if whenever $\theta_1\in \AA$ and $\theta_2$ is another non-negative integer valued radon measure on $(E,\FF)$ such that $\theta_1(A)\le \theta_2(A)$ for all $A\in \FF$, then $\theta_2\in \AA$ .
If $\XX=\del_{X_1}+\ldots +\del_{X_n}$ and $\YY=\del_{Y_1}+\ldots +\del_{Y_m}$ are two point processes on $E$, we say that $\XX$ is stochastically dominated by $\YY$ if $\P\{\XX\in \AA\}\le \P\{\YY\in \AA\}$ for any increasing set $\AA$.
\begin{theorem}\label{thm:determinantaldomination} Let $\XX_1$ and $\XX_2$ be  determinantal point processes on $(X,\mu)$ with finite  kernels $K_1(x,y)=\sum_{k=1}^n\phi_k(x)\bar{\phi_k}(y)$ and $K_2(x,y)=\sum_{k=1}^{n+1}\phi_k(x)\bar{\phi_k}(y)$, where $\phi_1,\ldots ,\phi_{n+1}$ is an orthonormal set in $L^2(E,\mu)$. Then, $\XX_1$ is stochastically dominated by $\XX_2$.
\end{theorem}
This theorem is due to Russell Lyons (see Theorem~6.2 and Theorem~7.1 in \cite{Lyons}) in the discrete case. There have been extensions of it in various ways, for example, \cite{goldman} and \cite{LyonsICM}, but the conditions there are restrictive. Our proof is essentially the same as that of Lyons, but written in such a way that the validity in the general situation is clear. The main difficulty in literally transcribing his proof is that $\del_x$ is not an element of $L^2(E,\mu)$ when $\mu$ is not discrete. By moving away from the exterior algebra language employed by Lyons, and writing everything in terms of determinants, this issue can be avoided.

In order to make the exposition clearer, we will first prove Theorem \ref{thm:determinantaldomination} in the discrete setting but the proof for general case is exactly the same with obvious modifications. 

\section{Stochastic domination: the discrete case}
\para{Discrete determinantal measures} Let $E=\{1,2,\ldots \}$ and let $\phi_{1},\ldots ,\phi_{n+1}$ be orthonormal in $\ell^{2}(E)$. The matrices
\ba
M=\l[ \begin{array}{cccc} 
	\phi_{1}(1) &\phi_{1}(2) & \ldots & \ldots \\
	\vdots & \vdots & \ldots & \ldots \\
	\phi_{n+1}(1) &\phi_{n+1}(2) & \ldots & \ldots
\end{array} \r] \;\;\; \mb{ and } \;\;\;
Q=\l[ \begin{array}{cccc} 
	\phi_{1}(1) &\phi_{1}(2) & \ldots & \ldots \\
	\vdots & \vdots & \ldots & \ldots \\
	\phi_{n}(1) &\phi_{n}(2) & \ldots & \ldots
\end{array} \r]
\ea
satisfy $MM^{*}=I_{n+1}$ and $QQ^{*}=I_{n}$. For a subset $A\subseteq E$, by $M_{A}$ (or $Q_A$) we mean the submatrix of $M$ (or $Q$) got by choosing the columns of $M$ indexed by elements of $A$ (keeping the order of rows and columns same as in the matrix $M$ (or $Q$)). Let $E^{\wedge k}$ denote the set of $k$-element subsets of $E$. The probability measures given by
\ba
\P_1(A)&=|\det(Q_{A})|^{2}  \;\;\; \mb{ for }A\in E^{\wedge n}, \\
\P_2(B)&=|\det(M_{B})|^{2}   \;\;\; \mb{ for }B\in E^{\wedge (n+1)}. 
\ea
are determinantal with kernel $K_1(x,y)=\summ_{i=1}^{n}\phi(x)\bar{\phi}_i(y)$ and $K_2(x,y)=\summ_{i=1}^{n+1}\phi(x)\bar{\phi}_i(y)$ respectively. The Cauchy-Binet formula shows that $\P_1$ and $\P_2$ are probability measures. Note that $\P_1$ and $\P_2$ can be extended as the probability measures on the power set of $E,$ that is, on the set $2^E$ by setting $\P_1(A)=0$ for any $A\in 2^E$ with $|A|\neq n$ and similarly $\P_2(B)=0$ if $|B|\neq (n+1).$ Let $X\subseteq E$ be a set chosen according to $\P_1$ and $Y\subseteq E$ be chosen according to $\P_2,$ and let $\XX=\sum\limits_{x\in X}\delta_{x}$ be the point process associated with $X$ and similarly let $\YY$ be the point process associated with $Y.$ It is clear that $\XX$ and $\YY$ are the determinantal processes associated with the kernel $K_1$ and $K_2$ respectively. 

The goal is to compare these two determinantal processes. It would be useful, however, to think of these point processes in terms of random subsets instead of random measures. And, we translate the Lyons' theorem in terms of subsets of $E$ without any mention of $\XX$ and $\YY.$ Before we do that, let us make a simple observation which will motivate our upcoming notations. Let $\mathcal{E}$ be an increasing subset of $\mathcal{M}(E).$ Since we will be  interested in the probability $\P(\XX\in \mathcal{E})$ and $\P(\YY\in \mathcal{E}),$ let us analyze these carefully. As $\XX$ is a simple point process if $\XX=\mu\in \mathcal{E}$ then $\mu$ can be associated to a unique subset $A\subseteq E$ of cardinality $n.$ Thus $\P(\XX\in \mathcal{E})=\P_1(\mathcal{E}_0)$ where $\mathcal{E}_0:=\{A\in E^{\wedge n}: \mu_A\in\mathcal{E}\}.$ In a similar way, we also get that $\P(\YY\in \mathcal{E})=\P_2(\mathcal{E}_1)$ where  $\mathcal{E}_1:=\{A\in E^{\wedge (n+1)}: \mu_A\in\mathcal{E}\}.$ This shows us how the probabilities like $\P(\XX\in \mathcal{E}), \P(\YY\in \mathcal{E})$ can be recast in terms of $\P_1$ and $\P_2.$ We now try to understand what conditions on $\mathcal{E}_0$ and $\mathcal{E}_1$ translate to the condition that $\mathcal{E}$ is increasing. To this end, let $\mathcal{E}_0$ be associated to $\mathcal{E}$ as above and let $A\in \mathcal{E}_0.$ Then $\mu_A\prec \mu_B$ if and only if $A\subseteq B.$ Therefore if $\mathcal{E}$ is increasing and $\mu_A\in \mathcal{E}$ then $\mu_B\in \mathcal{E}$ for all $A\subseteq B.$ We also point out that this entails that if $\mathcal{E}$ is increasing and $A\in \mathcal{E}_0$ then $A\cup\{x\}\in \mathcal{E}_1$ for all $x\in E\setminus A.$ With this discussion, we are now ready to translate the Lyons' theorem.

Let us fix the following notations. By  $M_{A|x}$ we will denote the matrix that has the same columns as  $M_{A\cup\{x\}}$, except that the column corresponding to $x$ is placed at the end. For $x\in E$ and $A\subseteq E$ we define $r(A,x)=|\{y\in \AA:y>x\}|$. If $\AA\subseteq E^{\wedge n}$ and $\BB\subseteq E^{\wedge (n+1)}$, then we say that $\AA\le \BB$ if $A\cup \{x\}\in \BB$ for any $A\in \AA$ and any $x\in E\setminus A$. Then Lyons' theorem on stochastic domination can be stated in this setting as follows. 

\begin{theorem}\label{thm:stochdom} Suppose  $\AA\subseteq E^{\wedge n}$ and $\BB\subseteq E^{\wedge (n+1)}$. If   $\AA\le \BB$, then $\P_1(\AA)\le \P_2(\BB)$. 
\end{theorem}

The proof of the above theorem will require two results. Note that we write $\P_1(\AA)$ in terms of determinants of submatrices of $Q,$ while $\P_2(\BB)$ is written in terms of determinant of submatrices of $M.$ It is but natural to obtain a way to relate the determinant of a submatrix of $Q$ to that of a submatrix of $M.$ The following proposition serves the purpose. 

\begin{proposition}\label{claim:detequality} For any $A\in E^{\wedge n}$, we have \[\summ_{x\not\in A}(-1)^{r(A,x)}\overline{\phi}_{n+1}(x)\det(M_{A\cup\{x\}})=\det(Q_{A}),\] where $r(A,x)=|\{k\in A: k>x\}|$. 
\end{proposition}
\begin{proof} Let $A=\{1,2,\ldots ,n\}$ without loss of generality. As $(-1)^{r(A,x)}\det(M_{A\cup\{x\}})=\det(M_{A|x})$,  the summand on the left hand side is $\overline{\phi}_{n+1}(x)\det(M_{A|x})$. The sum can be extended to all $x\in E$, since $\det(M_{A|x})=0$ for $x\in A$. Thus the sum on the left is equal to   \footnotesize
	\begin{align*}
	\sum_{x\in E}&\overline{\phi}_{n+1}(x) \det\mat{Q_{A}}{\begin{array}{c}\phi_{1}(x) \\ \vdots \\ \phi_{n}(x) \end{array}}{\begin{array}{lll} \phi_{n+1}(1) & \ldots & \phi_{n+1}(n) \end{array}}{\phi_{n+1}(x)}  \\
	&= \det\mat{Q_{A}}{\begin{array}{c}\<\phi_{1},\phi_{n+1}\> \\ \vdots \\ \<\phi_{n},\phi_{n+1}\> \end{array}}{\begin{array}{lll} \phi_{n+1}(1) & \ldots & \phi_{n+1}(n) \end{array}}{\<\phi_{n+1},\phi_{n+1}\>}
	\end{align*}
	\normalsize 
	by  multilinearity of the determinant. As $\phi_j$ are orthonormal, the last column is  $(0,\ldots ,0,1)^{t}$. Hence the determinant is equal to $\det(Q_A)$.
\end{proof}

\normalsize

\begin{lemma}
	Let $\phi: E\to \mathbb{C}$ and $\epsilon: E\times E^{\wedge n}\to \{+1,-1\}$ be any arbitrary functions. Let $\AA\subseteq E^{\wedge n}.$ Let $\mathcal{M}$ be the matrix (with rows and columns indexed by the elements of $\AA$) given by 
	$$
	\mathcal{M}(A,C)=
	\left\{
	\begin{array}{ccccc}
	\summ_{x\in A}|\phi(x)|^2, & \text{if}\hspace{1mm}A=C\\
	\epsilon(x,A)\epsilon(y,C)\phi(x)\overline{\phi(y)}, & \stackrel{\text{if}\hspace{1mm} |A\cap C|=n-1, \text{where}}{ x\in A\setminus C, y\in C\setminus A} \\
	0, & \text{otherwise}\\
	\end{array}.
	\right.
	$$
	Then, the matrix $\mathcal{M}$ is positive semidefinite. 
	\label{positivitylemma}
\end{lemma}

\begin{proof}
	Let us consider the matrix $X$ (with rows indexed by $\AA$, and columns indexed by $E^{\wedge (n-1)}$) defined by
	$$X(A,T)=\left\{ \begin{array}{cccccc}
	\epsilon(x,A)\phi(x), & \text{if}\hspace{1mm} T\subset A \hspace{1mm}\text{and}\hspace{1mm}\{x\}=A\setminus T \\
	0, & \text{otherwise}
	\end{array}.
	\right.$$
	
	Observe that \begin{align*}
	XX^*(A,A)&=\summ_{T \in E^{\wedge n-1}}X(A,T)X^*(T,A)\\
	&=\summ_{T\subset A:|T|=n-1}|\phi(x)|^2 \\
	&= \summ_{x\in A}|\phi(x)|^2\\
	&= \mathcal{M}(A,A).
	\end{align*}
	Clearly, when $|A\cap C|\le n-2$ then $XX^*(A,C)$ is zero. A similar computation shows that when $|A\cap C|=n-1$ then,
	\begin{align*}
	XX^*(A,C)&= \summ_{T\in E^{\wedge n-1}}X(A,T)X^*(T,C)\\
	&= \epsilon(x,A)\epsilon(y,C)\phi(x)\overline{\phi}(y), \hspace{2mm}\text{where}\hspace{1mm} x\in A\setminus C, y\in C\setminus A\\
	&= \mathcal{M}(A,C).
	\end{align*}
	This proves that the matrix $\mathcal{M}=XX^*$ and hence positive semidefinite.
\end{proof}
\remark 
Observe that in the proof of the lemma \ref{positivitylemma} we do not use any special property of $\phi.$ It is true for any $\phi$ and any $\epsilon$. 

\begin{remark}
Note that as a result of the above lemma, we get that for any function $F:\AA\to \C$ we have $\langle X^*F,F\rangle\ge 0$ (Here the inner product is taken with in $\ell^2(\AA).$ As $\AA$ is at most countable there is a natural way to equip $\ell^2(\AA)$ with an inner product. Let $F, G:\AA\to \C$ be two functions then $\langle F, G\rangle:=\sum\limits_{A\in \AA}F(A)\overline{G(A)}.$)

This is the way it would be used later. In the next section where we prove the stochastic domination in continuous setting, we do not record it as a separate lemma but it is used implicitly in one step.
\end{remark}

\bprf[Proof of Theorem~\ref{thm:stochdom}] We shall write $\phi$ for $\phi_{n+1}.$
\begin{align*}
\P_1(\AA)=\summ_{A\in \AA}|\det(Q_A)|^2&= \summ_{A\in \AA}\det(Q_A)\overline{\det(Q_A)}\\
&= \summ_{A\in \AA}\summ_{x\notin A}(-1)^{r(A,x)}\overline{\phi(x)}\det(M_{A\cup \{x\}})\overline{\det(Q_A)}\\
&=\summ_{B\in \BB}\det(M_B)\summ_{x:B\setminus\{x\}\in \AA}(-1)^{r(B\setminus\{x\},x)}\overline{\phi(x)}\overline{\det(Q_{B\setminus\{x\}})}.
\end{align*}
By Cauchy-Schwarz inequality we get that 
\ba 
\P_1(\AA) \le\left( \summ_{B\in \BB}|\det(M_B)|^2\right)^{\frac{1}{2}} \left(\summ_{B\in \BB}\left|\summ_{x: B\setminus\{x\}\in \AA}(-1)^{r(B\setminus\{x\},x)}\overline{\phi(x)}\overline{\det(Q_{B\setminus\{x\}})}\right|^2\right)^{\frac{1}{2}}.
\ea

Now observe that for a fixed $B\in \BB,$ we have the following
\begin{align*}
&\left\lvert\summ_{x: B\setminus\{x\}\in \AA}(-1)^{r(B\setminus\{x\},x)}\overline{\phi(x)}\overline{\det(Q_{B\setminus\{x\}})}\right\rvert^2 \\
&=\summ_{x,y: B\setminus\{x\},B\setminus\{y\}\in \AA}(-1)^{r(B\setminus\{x\},x)}\overline{\phi(x)}\overline{\det(Q_{B\setminus\{x\}})}(-1)^{r(B\setminus\{y\},y)}\phi(y)\det(Q_{B\setminus\{y\}})\\
&= \summ_{\stackrel{A,C\in \AA:}{ A\cup C \subseteq B}}(-1)^{r(A,x)}\phi(x)\det(Q_{A})(-1)^{r(C,y)}\overline{\phi(y)}\overline{\det(Q_{C})},
\end{align*}
where $x,y$ are the unique elements such that $x\in B\setminus A$ and $y\in B\setminus C.$ Therefore we can write the above expression as
\footnotesize
\begin{align*}
&\summ_{B\in \BB}\summ_{A,C\in \AA: A\cup C \subseteq B}(-1)^{r(A,x)}\overline{\phi(x)}\overline{\det(Q_{A})}(-1)^{r(C,y)}\phi(y)\det(Q_{C})\\
&=\summ_{A,C\in \AA}\det(Q_A)\overline{\det(Q_C)}\summ_{x\notin A, y\notin C:A\cup\{x\}=C\cup\{y\}}(-1)^{r(A,x)+r(C,y)}\phi(x)\overline{\phi(y)}.
\end{align*}
\normalsize
Note that when $A=C$ the inner sum becomes $\summ_{x\notin A}|\phi(x)|^2$, when $A\neq C$ the inner sum is non-empty precisely when $|A\cap C|=n-1.$ Therefore, we write the above sum as
\footnotesize 
\begin{align*}
&\summ_{A\in\AA}|\det(Q_A)|^2\summ_{x\notin A}|\phi(x)|^2+\summ_{|A\cap C|=n-1}(-1)^{r(A,x)+r(C,y)}\det(Q_A)\overline{\det}(Q_C)\phi(x)\overline{\phi}(y)\\
&= \summ_{A\in \AA}|\det(Q_A)|^2-\\
&\left(\summ_{A\in \AA}\det(Q_A)\sum_{x\in A}|\phi(x)|^2 - \summ_{|A\cap C|=n-1}(-1)^{r(A,x)+r(C,y)}\det(Q_A)\overline{\det(Q_C)}\phi(x)\overline{\phi(y)}\right)\\
&= \summ_{A\in \AA}|\det(Q_A)|^2-\\
&\left(\summ_{A\in \AA}\det(Q_A)\sum_{x\in A}|\phi(x)|^2 +\summ_{|A\cap C|=n-1}(-1)^{r(A,y)+r(C,x)}\det(Q_A)\overline{\det(Q_C)}\phi(x)\overline{\phi(y)}\right).
\end{align*}
\normalsize
In the last equality, we used the fact that $(-1)^{r(A,x)+r(C,y)}=-1(-1)^{r(A,y)+r(C,y)}.$
The theorem follows, if we can show that the quantity in the bracket above is positive. To this end define a function $F:\AA\to \C$ by $F(A)=\det(Q_A)$ and observe that the quantity in the bracket is nothing but $\langle \mathcal{M}F, F\rangle$ where $\mathcal{M}$ is the matrix (with rows and columns indexed by the elements of $\AA$) defined by $$
	\mathcal{M}(A,C)=
	\left\{
	\begin{array}{ccccc}
	\summ_{x\in A}|\phi(x)|^2, & \text{if}\hspace{1mm}A=C\\
	(-1)^{r(y,A)}(-1)^{r(x, C)}\phi(x)\overline{\phi(y)}, & \stackrel{\text{if}\hspace{1mm} |A\cap C|=n-1, \text{where}}{ x\in A\setminus C, y\in C\setminus A} \\
	0, & \text{otherwise}\\
	\end{array}.
	\right.
	$$
It follows from lemma \ref{positivitylemma} that $\mathcal{M}$ is positive definite and hence $\langle \mathcal{M}F, F\rangle$ is positive, which completes the proof. 
\eprf

\remark 
It is natural at this stage to ask if we have a similar result for bi-orthogonal ensemble. Recall that $$\P_n(dx_1,\ldots, dx_n):= C_n\det[\phi_i(x_j)]_{i,j=1}^{n}\det[\psi_i(x_j)]_{i,j=1}^{n}\prod\limits_{i=1}^{n}\mu(dx_i).$$
for suitable normalization constant $C_n>0$, and function $\phi_i, \psi_i$ such that all the integrals $G_{ij}:=\int \phi_i(x)\psi_j(x)\mu(dx)$ are finite, defines a determinantal probability measure. One can naturally ask if $\P_n\prec \P_{n+1}$ in this case. The answer to this question is, 'No'. A fairly simple counter-example can be constructed as follows.
 
 \beg	
 Consider the set $E={a,b,c}$ equipped with the uniform probability measure $\mu.$ Now, let's define the functions $\phi_i, \psi_i, i=1,2$ on $E$ as follows:
 $$\phi_1(a)=\phi_1(b)=1, \phi_1(c)=0$$
 $$\phi_2(a)=\phi_2(b)=1, \phi_2(c)=1$$
 $$\psi_1(a)=1, \psi_1(b)=0, \psi_1(c)=-1$$
 $$\psi_2(a)=-1, \psi_2(b)=1=\psi_2(c)$$
 Observe that $\langle \phi_i,\psi_j\rangle =\delta_{ij}$ which means the kernel $K_n(x,y)=\sum_{i,j=1}^n\phi_i(x)\psi_j(y).$ 
 
 We will compare the determinantal processes with kernel $K_1$ and $K_2$ (Let us call the corresponding probability measures as $\P_1, \P_2$ respectively.) 
 
 For $n=1,$ we see that the kernel $K_1(x,y)=\phi_1(x)\psi_1(y).$ Recall that the point process $\X_1$ defined by the kernel $K_1$ has exactly one point almost surely. We thus obtain that $\P_1(x)=K_1(x,x)=\phi_1(x)\psi_1(y)$ which gives us that $\P_1(\{a\})=1, \P_1(\{b\})=0=\P_1(\{c\}).$ 
 
 For $n=2$, similarly, the point process $\X_2$ has exactly two points almost surely. Therefore, it suffices to compute the probability of each subset of $E$ which has cardinality $2$. Recall that $\P_2(\{x,y\})=\det\left(\begin{matrix}
 \phi_1(x) & \phi_1(y)\\
 \phi_2(x) & \phi_2(y)
 \end{matrix}\right)\det \left(\begin{matrix}
 \psi_1(x) & \psi_1(y)\\
 \psi_2(x) & \psi_2(y)
 \end{matrix}\right).$ Using which let us calculate all the relevant probabilities.
 
 $$\P_2(\{a,b\})=\det\left(\begin{matrix}
 \phi_1(a) & \phi_1(b)\\
 \phi_2(a) & \phi_2(b)
 \end{matrix}\right)\det \left(\begin{matrix}
 \psi_1(a) & \psi_1(b)\\
 \psi_2(a) & \psi_2(b)
 \end{matrix}\right))=0,$$
 
 $$\P_2(\{a,c\})=\det\left(\begin{matrix}
 \phi_1(a) & \phi_1(c)\\
 \phi_2(a) & \phi_2(c)
 \end{matrix}\right)\det \left(\begin{matrix}
 \psi_1(a) & \psi_1(c)\\
 \psi_2(a) & \psi_2(c)
 \end{matrix}\right)=0.$$
 And, 
 $$\P_2(\{b,c\})=\det\left(\begin{matrix}
 \phi_1(b) & \phi_1(c)\\
 \phi_2(b) & \phi_2(c)
 \end{matrix}\right)\det \left(\begin{matrix}
 \psi_1(b) & \psi_1(c)\\
 \psi_2(b) & \psi_2(c)
 \end{matrix}\right)=1.$$
 
 If we start with $\AA=\{\{a\}\}\subset E^{\wedge 1}$ and let $\{\{a,b \}, \{a,c\}\}=\BB\subset E^{\wedge 2}$ then $\AA \prec \BB$ but $\P_1(\AA)=1\not\le 0=\P_2(\BB).$
 \eeg
 \section{Stochastic domination: General finite rank case}
Now let $(E,\FF,\mu)$ be a measure space and let $\phi_{1},\ldots ,\phi_{n+1}$ be an orthonormal set. Let $X=(X_{1},\ldots ,X_{n})$ be a random vector taking values in $E^{n}$ and having density (w.r.t.  $\mu^{\otimes n}$)
\ba
\frac{1}{n!}|\det(\phi_{i}(x_{j}))_{i,j\le n}|^{2}.
\ea
The determinantal process corresponding to this measure is defined to be the random set $\XX=\{X_{1},\ldots ,X_{n}\}$ (or as the random measure $\del_{X_{1}}+\ldots +\del_{X_{n}}$ which is sometimes more convenient).  Similarly, define $Y=(Y_{1},\ldots ,Y_{n+1})$ to be a random vector taking values in $E^{\wedge n+1}$ and having density (w.r.t. $\mu^{\otimes (n+1)}$) and let $\YY$ be the determinantal process corresponding to this measure. Since the density of $X$ vanishes unless $x_{i}$s are distinct, it is clear that $\XX$ takes values in the collection of $n$-element subsets of $E$. But it is clear that everything about $\XX$ can also be formulated in terms of the random vector $X$ and that is what we do here. Henceforth we do not mention $\XX$ or $\YY$.

Let $\AA\subseteq E^{n}$ be a measurable subset (i.e., in $\FF^{\otimes n}$) that is symmetric (i.e., closed w.r.t. permutation of co-ordinates). Similarly let $\BB$ be a measurable symmetric subset of $E^{n+1}$. Then we say that $\AA\le \BB$ if $(x,\ldots ,x_{n},t)\in \BB$ for any $(x_{1},\ldots ,x_{n})\in \AA$ and any $t\in E$.
\begin{theorem}\label{thm:stochdominationgeneral} Let $\AA$ and $\BB$ be measurable, symmetric subsets of $E^{n}$ and $E^{n+1}$, respectively. Assume that $\AA\le \BB$. Then $\P\{X\in \AA\}\le \P\{Y\in \BB\}$.  
\end{theorem}
As before, we shall need two claims, analogous to the discrete situation (except that each set of $n$ elements is replaced by $n!$ tuples). Let us fix the following notation. For $x=(x_1,x_2,\hdots, x_n)\in E^n,$ define $K_n(x):= (\phi_i(x_j))_{1\le i,j\le n}.$ Also, for $t\in E, x\in E^n$ we will write $(x|t;n+1)=(x|t):=(x_1,x_2,\hdots, x_n, t)\in E^{(n+1)}$ and if $k\in [n]$ then define the vector $(x|t;k)$ to be the vector obtained by putting $t$ at the $k$-th coordinate in $x$, that is, $(x|t;k)=(x_1,\ldots,x_{k-1},t,x_{k+1},\ldots,x_n)$. We begin with the following claim:

\begin{proposition}\label{claim:detequalitygeneral} For any $(x_{1},\ldots ,x_{n})\in E^{n}$, we have 
	\ba
	\int_{E}\bar{\phi}_{n+1}(t) \times \det(K_{n+1}(x|t))\ d\mu(t) = \det(K_n(x)).
	\ea 
\end{proposition}
\begin{proof} By the multilinearity of the determinant, the integral becomes (inner products in $L^{2}(\mu)$)
	\ba
	\det\mat{(\phi_{i}(x_{j}))_{i,j\le n}}{\begin{array}{c}\<\phi_{1},\phi_{n+1}\> \\ \vdots \\ \<\phi_{n},\phi_{n+1}\> \end{array}}{\begin{array}{lll} \phi_{n+1}(x_1) & \ldots & \phi_{n+1}(x_n) \end{array}}{\<\phi_{n+1},\phi_{n+1}\>} .
	\ea
	But then the last column is $(0,\ldots ,0,1)^{t}$, hence we get $\det(\phi_{i}(x_{j}))_{i,j\le n}$.
\end{proof}
\begin{remark}
	Note that the above proposition is entirely analogous to Proposition \ref{claim:detequality}. We will now prove a lemma which analogous to the lemma \ref{positivitylemma} but here we directly prove what we would use it for. 
\end{remark}

Let us fix the following notation for the next proof. For $y=(y_{1},\ldots ,y_{n+1})$, let $\hat{y}_{k}=(y_{1},\ldots ,y_{k-1},y_{k+1},\ldots ,y_{n})$.

\begin{lemma}\label{claim:detinequalitygeneral} Let $\AA$ be measurable, symmetric subsets of $E^{n}.$ Then,
	\ba
	\frac{1}{n!}\int_{\AA} |\det(K_n(x))|^{2}&d\mu^{\otimes n}(x) \\
	&\ge \frac{1}{(n+1)!}\int_{E^{n+1}} \Big| \sum_{k: \ \hat{y}_{k}\in \AA}(-1)^k \phi_{n+1}(y_{k})\det(K_{n}(\hat{y}_k)) \Big|^{2}d\mu^{\otimes (n+1)}(y)
	\ea.
\end{lemma}

\begin{proof}
	First observe that
	\ba 
	\Big|& \sum_{k:\hat{y}_{k}\in \AA}(-1)^k \phi_{n+1}(y_{k})\det(K_{n}(\hat{y}_k)) \Big|^{2}\\ &= \summ_{k=1}^{n+1}|\det(K_n(\hat{y}_k))|^2|\phi(y_k)|^2\chi_{\AA}(\hat{y}_k)\\
	&+\summ_{j,k=1}^{n+1}(-1)^{j+k}\phi_{n+1}(y_k)\overline{\phi_{n+1}}(y_j)\det(K_n(\hat{y}_k))\overline{\det}(K_{n}(\hat{y}_j))\chi_{\AA}(\hat{y}_k)\chi_{\AA}(\hat{y}_j).
	\ea
	Now note that $$\intt_{E^{n+1}}\det(K_n(\hat{y}_k))|^2|\phi(y_k)|^2\chi_{\AA}(\hat{y}_k)=\intt_{\AA}|\det(K_n(x))|^2d\mu^{\otimes n}(x)\intt_E|\phi(t)|^2d\mu(t).$$
	And therefore we get that 
	
	\ba
	&\frac{1}{(n+1)!}\int_{E^{n+1}} \Big| \sum_{k:\hat{y}_{k}\in \AA}\eps(y,k) \phi_{n+1}(y_{k})\det(K_{n}(\hat{y}_k)) \Big|^{2}d\mu^{\otimes (n+1)}(y) \\
	&=\frac{1}{n!}\intt_{\AA}|\det(K_n)(x)|^2d\mu^{\otimes n}(x)\\ &-\frac{1}{(n+1)!}\summ_{j,k=1}^{n+1}\intt_{E^{n+1}}(-1)^{j+k-1}\phi_{n+1}(y_k)\overline{\phi_{n+1}}(y_j)\det(K_n(\hat{y}_k))\overline{\det}(K_{n}(\hat{y}_j))\chi_{\AA}(\hat{y}_k)\chi_{\AA}(\hat{y}_j).
	\ea

Note that it suffices to show that
$$\intt_{E^{n+1}}(-1)^{j+k-1}\phi_{n+1}(y_k)\overline{\phi_{n+1}}(y_j)\det(K_n(\hat{y}_k))\overline{\det}(K_{n}(\hat{y}_j))\chi_{\AA}(\hat{y}_k)\chi_{\AA}(\hat{y}_j)\ge 0.$$ 

Let us denote by $\tilde{\AA}=\{\hat{y}_k: y\in \AA\}$ (Note that $\tilde{\AA}$ is well defined i.e. independent of $k$ due to the symmetry of $\AA.$) Let $\AA_0=\{y_1: y\in \AA\}$. Now, note that a fixed $k, j$ and a vector $y\in E^{n+1}$ is such that $\hat{y}_k\in \AA$ and $\hat{y}_j\in \AA$ corresponds uniquely to a triplet $(x,t_1,t_2)$ where $x\in \tilde{\AA}$ and $t_1,t_2\in \AA_0$ (We obtain $x$ by dropping both $y_j$ and $y_k$ from $y$ and say $t_1=y_k$ while $t_2=y_j$). Therefore, we rewrite the above integral as
\ba
\intt_{\tilde{\AA}}\intt_{\AA_0}\intt_{\AA_0}\det(K_n(x|t_1))\overline{\det}(K_n(x|t_2))\phi_{n+1}(t_1)\overline{\phi_{n+1}}(t_2)d\mu(t_1)d\mu(t_2)d\mu^{\otimes (n-1)}(x).
\ea 

To show that the above integral is positive, we show that the above integral is norm square of some function, and therefore non-negative. To this end, define an operator $T:L^2(\AA_0)\to L^2(\tilde{\AA})$ by
$$Tf(x)=\intt_{A_0}\det(K_n)(x|t)d\mu(t).$$
Observe that 
\ba
0 &\le \langle Tf, Tf\rangle\\
&=\intt_{\AA_1}\intt_{A_0}\intt_{A_0}\det(K_n(x|t_1))\overline{\det}(K_n(x|t_2))\phi_{n+1}(t_1)\overline{\phi_{n+1}}(t_2)d\mu(t_1)d\mu(t_2)d\mu^{\otimes(n-1)}(x),
\ea
which completes the proof.
\end{proof}

\begin{proof}[Proof of Theorem~\ref{thm:stochdominationgeneral}] Let $p_{1}=\P\{X\in \AA\}$ and $p_{2}=\P\{Y\in \BB\}$. Then,
	\ba
	p_{2}&=\frac{1}{(n+1)!}\intt_{\BB}|\det(K_{n+1}(y))|^{2}d\mu^{\otimes n+1}(y),\hspace{10mm}\text{and} \\
	p_{1}&\ge \frac{1}{(n+1)!}\intt_{\BB}\Big|\sum_{k:\ \hat{y}_{k}\in \AA} \eps(y,k) \phi_{n+1}(y_{k})\det(K_n(\hat{y}_{k})) \Big|^{2}d\mu^{\otimes(n+1)}(y).
	\ea
	where the second line follows from Claim~\ref{claim:detinequalitygeneral}. Now use Cauchy-Schwarz inequality to write
	\ba
	\sqrt{p_{1}}\sqrt{p_{2}} &\ge \frac{1}{(n+1)!} \intt_{\BB} \overline{\det(K_{n+1}(y))} \sum_{k:\ \hat{y}_{k}\in \AA}\eps(y,k) \phi_{n+1}(y_{k})\det(K_n(\hat{y}_k))d\mu^{\otimes(n+1)}(y).
	\ea
	Choose $\eps(y,k)$ so that $\eps(y,k)\det(K_{n+1}(y))=\det(K_n)(\hat{y}_{k}|y_{k})$ (in simpler words, $\eps(y,k)=(-1)^{n-k+1}$).
	
	Now fix $x\in \AA$ and $t\in E$. Since $\AA\le \BB$, for each $k$ there is a unique $y\in \BB$ such that $\hat{y}_{k}=x$ and $y_{k}=t$.  Then, each $k$ contributes the same, and we get
	\ba
	\sqrt{p_{1}}\sqrt{p_{2}} &\ge \frac{1}{n!}\intt_{\AA} \det(K_n(x)) \intt_{E}\phi_{n+1}(t) \overline{\det(K_{n+1})}(x,t)d\mu(t)d\mu^{\otimes n}(x).
	\ea
	The inner integral is equal to $\overline{\det K_n }(x)$, by Claim~\ref{claim:detequalitygeneral}. Thus we arrive at $\sqrt{p_{1}}\sqrt{p_{2}}\ge p_{1}$, which proves that $p_{2}\ge p_{1}$.
\end{proof}

%% file: chapter4.tex
\chapter{Another result on Stochastic domination}
Before we go to our next result, we must point out that the content of the Lyons' theorem (proved in the last chapter) is that `an orthogonal projection on bigger space gives larger determinantal measure'. Lyons' theorem allows us to compare two determinantal measures (obtained from finite rank projection kernels) whose kernels are expressed with respect to the same measure. Now suppose that we have two determinantal probability measures coming from orthogonal projections of the $\mbox{span}\{1,x,\ldots,x^{n-1}\}$, but with respect to two different reference measures. That is, let $\mathcal{H}_1=\mbox{span}\{1,x,\ldots,x^{n-1}\}\subseteq L^2(\mu_1)$ and $\mathcal{H}_2=\mbox{span}\{1,x,\ldots,x^{n-1}\}\subseteq L^2(\mu_2).$ In this case, is there a reasonable way to compare the determinantal processes coming from orthogonal projections on $\mathcal{H}_1$ and $\mathcal{H}_2?$ We answer a variant of this question in the following section. 
\section{Another stochastic domination result}
As a prelude, we begin with the following proposition.
\begin{proposition}
	Let $\mu$ be a positive measure on $\mathbb{R},$ and let $f,g$ be two non-negative integrable functions on $\mathbb{R}$ such that $\int\limits_{\mathbb{R}}f=\int\limits_{\mathbb{R}}g=1,$ and $\frac{f}{g}$ is increasing. Then for any real $t$ we have
	$$\int\limits_{-\infty}^{t}fd\mu\le \int\limits_{-\infty}^{t}gd\mu.$$
\end{proposition}

It should be pointed out the above theorem is essentially a result about stochastic domination of two probability measures. It is standard in measure theory to induce positive measures $\mu_f$ from a positive functions $f$ by defining $d\mu_f=f\;d\mu$. The integral of the functions being $1$ ensures that we obtain a probability measure and the content of the above theorem can be written as $\mu_g\prec \mu_f$ if $\frac{f}{g}$ is increasing. 
We will prove the following above proposition in slightly general setting, that is, when $f, g$ are densities given on some totally ordered measure space. The above result can then be obtained as a corollary. Note that if $(E, \le)$ is a totally ordered set, we say that a function $f:E\to \R$ is increasing if $f(x)\le f(y)$ whenever $x\le y.$ Similarly, we say $\mathcal{A}\subset E$ is increasing if $y\in \mathcal{A}$ whenever $x\le y$ for some $x\in \mathcal{A}.$  

\begin{proposition}
	Let $(E,\mu, \le)$ be a totally ordered probability space (that is, $(E, \le)$ is a totally ordered set). Let $h:E\to \mathbb{R}$ be a probability density with respect to $\mu$ which is increasing. Then for any increasing subset $\AA$ of $E,$ we have
	$$\mu(\AA)\le \mu_h(\AA):=\int\limits_{\AA}hd\mu.$$
\end{proposition}
\begin{proof}
Consider the set $S:=\{x\in E: h(x)\ge 1\}\subset E.$ Observe that $S$ is an increasing subset of $E.$ It is clear that if $\AA\subset S$ then $\mu(\AA)\le\mu_h(\AA)$ since $h\ge 1$ on $S.$ Similarly for any subset $B\subset S^c$ we have that $\mu_h(B)\le \mu(B).$ Suppose, for the sake of contradiction, that $\AA\subset E$ be an increasing set such that 
	\begin{equation}
	\label{eq1}
	\mu_h(\AA)<\mu(\AA).
	\end{equation}
Note that $\AA^c \subset S^c.$ Therefore, $\mu_h(\AA^c)\le \mu(\AA^c).$ Adding this to equation \eqref{eq1} we get
	$$1=\mu_h(\AA)+\mu_h(\AA^c)<\mu(A)+\mu(\AA^c)=1.$$
which is a contradiction. Therefore, for any increasing set $\AA$ in $E$ we must have $\mu(\AA)\le \mu_h(\AA).$
\end{proof}
\begin{remark}
One can obtain the above result directly from Harris inequality whose proof usually goes by observing that $(h_1(x)-h_1(y))(h_2(x)-h_2(y))\ge 0,$ for any increasing functions $h_1, h_2,$ and therefore so its integral (with respect to a $d\mu_1(x)d\mu_2(y)$). In particular, taking $h_1=\mathbf{1}_{\AA}$ and $h_2=\frac{f}{g}$ and the measure to be $gd\mu$ we get an alternate proof of the above result.
\end{remark}

Note that the notion of increasing sets are available in partially ordered sets as well. It would be nice to obtain a result in the same spirit on a partially ordered set. But probably it is too good to be true. We produce below a counter-example which shows that the above result does not hold for an arbitrary partially ordered set. 

Consider the set $X=\{a,b,c\}$ equipped with the partial order $a\le b, a\le c.$ Let $\mu$ be uniform measure on $X$, that is, $\mu(\{x\})=\frac{1}{3}$ for every $x\in X.$
Now let $f:X\to \mathbb{R}$ be defined by $f(a)=\frac{1}{3}, f(b)=\frac{1}{2}, f(c)=\frac{13}{6}.$ Clearly $f$ is an increasing function on $X$ and is a probability density with respect to $\mu.$ The set $\{b\}\subset X$ is an increasing set, but $\frac{1}{6}= \mu_f(\{b\})<\mu(\{b\})=\frac{1}{3}.$

A simple modification of the above example also shows that the above result does not extend to a partially ordered lattice as well. Yet, not everything is lost. Our next result shows that we can obtain a stochastic domination between $\mu_f$ and $\mu_g$ at least under some conditions, which suffices for our purposes. Before we state our next result, we recall that the partial order on $X^{\wedge n}=$ (or $X^n$) is given by component wise ordering.

\begin{theorem}[]
	\label{thm:Stochasticdominationnew}
	Let $X=\mathbb{N}$ or $\R_{+}$  and let $\mu$ be a Borel (finite) measure such that $d\mu(x+y)=f(y)d\mu(x)$ for some positive function $f.$ 
	Let $X^{\wedge n}:=\{x=(x_1<x_2<\ldots<x_n)\},$ and let $H: X^{\wedge n}\to \mathbb{R}$ be an increasing function and consider the probability measures $P_1$ and $P_2$ on $X^{\wedge n}$ given by 
	$$dP_1(x)=\Delta(x)^2\prod_{i=1}^{n}d\mu(x_i)$$
	and,
	$$dP_2(x)=\Delta(x)^2H(x)\prod_{i=1}^{n}d\mu(x_i);$$
	where $\Delta(x)=\prod\limits_{i<j}(x_i-x_j).$ Let $\AA\subseteq X^{\wedge n}$ be an increasing set. Then $P_1(\AA)\le P_2(\AA).$
\end{theorem}

\begin{remark}
Note that in the statement of the theorem above the measure $d\mu$ and function $H$ are already suitably normalized. Also note that we can allow $H:X^n\to \mathbb{R}$ if $H$ is symmetric. We are dealing with $X^{\wedge n}$ instead of $X^n$ purely for the convenience, and with obvious modification one can write the above result in the alternate setting. 
\end{remark}
  
\begin{proof}[Proof of Theorem \ref{thm:Stochasticdominationnew}]
	We will prove the claim by induction on $n.$ For $n=1$ it follows from our previous result on Stochastic domination on totally ordered set. Assume the claim to be true for $n=m$ for some $m\ge 1,$ And let $n=m+1.$
	
	We first introduce some notations. Note that for $x\in X^{\wedge (m+1)}$ associate a vector $(t;r)\in X\times X^{\wedge m}$ given by $t=x_1, r_i:=x_{i+1}-x_1.$ We can then write $\Delta(x)^2=\Delta(r)^2\prod_{i=1}^{m}r_i^2.$ For future use we will also define $d\mu_1(s)=s^2d\mu(s)$ in order to simplify the notation in the proof.

	Also for an increasing set $\AA\subset X^{\wedge m+1}$ and $t\in X$ define $$\AA_t:=\{(d\in X^{\wedge m}: (t, t+d_1, \ldots, t+d_m) \in \AA\},$$ and observe that $\AA_t \subseteq \AA_s$ if $t\le s.$ 
	
	Now, observe that 
	\begin{align*}
	P_1(\AA) &:= \int_{\AA}\Delta(x)^2\prod_{i=1}^{m+1}d\mu(x_i)\\
	&=\int_{X}f(t)^md\mu(t)\int_{\AA_t}\Delta(r)^2\prod_{i=}^{m}d\mu_1(r_i)\\
	&= \int_{X}f(t)^mZ_1d\mu(t)\int_{\AA_t}\frac{\Delta(r)^2}{Z_1}\prod_{i=1}^{m}d\mu_1(r_i)
	\end{align*}
	where $Z_1:= \int_{X^{\wedge m}}\Delta(r)^2\prod_{i=1}^{m}d\mu_1(r_i).$ We note that $Z_1f(t)^md\mu(t)$ is a probability measure on $X,$ and $Z_1^{-1}\Delta(r)^2\prod_{i=1}^{m}d\mu_1(r_i)$ is a probability measure on $X^{\wedge m}.$
	
	Doing exactly the same for $P_2(\AA)$ we obtain that
	$$
	P_2(\AA)=\int_{X}f(t)^mZ_2(t)d\mu(t)\int_{\AA_t}\frac{H(t;r)}{Z_2(t)}\Delta(r)^2\prod_{i=1}^{m}d\mu_1(r_i)
	$$
	where $Z_2(t):=\int_{X^{\wedge m}}H(t;r)\Delta(r)^2\prod_{i=1}^{m}d\mu_1(r_i).$ (We are making a slight abuse of notation here we are using the same symbol $H$ for the function $\tilde{H}(t;r):=H(t,t+r_1, \ldots,t+r_m)$.) Observe that $Z_2(t)$ is increasing in $t$ and $f(t)^mZ_2(t)d\mu(t)$ is a probability measure on $X.$ 
	
	It follows therefore from our previous result ($n=1$ case) that the probability measure $dm_2(t):=f(t)^mZ_2(t)d\mu(t)$ stochastically dominates the probability measure $dm_1(t):=f(t)^md\mu(t).$ Therefore, we know that for any increasing function $F(t)$ we have that 
	\begin{equation}
	\int\limits_{X}F(t)dm_1(t)\le \int\limits_{X}F(t)dm_2(t).
	\label{induction_step}
	\end{equation}
	
	As $\AA_t$ is increasing in $t,$ we have $F(t):=\int\limits_{\AA_t}Z_1^{-1}\Delta(r)^2\prod_{i=1}^{m}d\mu_1(r_i)$ is increasing in $t,$ it follows from \eqref{induction_hypo} therefore that  
	\begin{equation}
	\int_{X}dm_1(t) \int\limits_{\AA_t}Z_1^{-1}\Delta(r)^2\prod_{i=1}^{m}d\mu_1(r_i)\le \int_{X}dm_2(t)\int\limits_{\AA_t}Z_1^{-1}\Delta(r)^2\prod_{i=1}^{m}d\mu_1(r_i).
	\label{second_last}
	\end{equation}
	
	We now observe that for a fixed $t,$ $\frac{H(t;r)}{Z_2(t)}$ is increasing in $r$ on $X^{\wedge m}.$ Therefore, it follows from induction hypothesis that for any increasing set $B\subseteq X^{\wedge m}$ we have that 
	$$\int\limits_{B}\Delta(r)^2\prod_{i=1}^{m}d\mu_1(r_i)\le
	\int\limits_{B}\Delta(r)^2\frac{H(t;r)}{Z_2(t)}\prod_{i=1}^{m}d\mu_1(r_i).$$
	(Note that the due to suitable normalization we have probability densities on both sides, which is crucial in order to apply induction.)
	In particular for $B=\AA_t,$ we get that 
	\begin{equation}
	F(t)=\int\limits_{\AA_t}\Delta(r)^2\prod_{i=1}^{m}d\mu_1(r_i)\le G(t) =:
	\int\limits_{\AA_t}\Delta(r)^2\frac{H(t;r)}{Z_2(t)}\prod_{i=1}^{m}d\mu_1(r_i).
	\label{induction_hypo}
	\end{equation}
	
	It follows from \eqref{induction_hypo} and the fact that $dm_2(t)$ is a positive measure that 
	
	\begin{equation}
	\int_{X}dm_2(t)\int\limits_{\AA_t}Z_1^{-1}\Delta(r)^2\prod_{i=1}^{m}d\mu_1(r_i)\le \int_{X}dm_2(t)\int\limits_{\AA_t}Z_2(t)^{-1}H(t;r)\Delta(r)^2\prod_{i=1}^{m}d\mu_1(r_i)
	\label{final}
	\end{equation}
	The proof follows from \eqref{second_last} and \eqref{final}.
\end{proof}

%% file: chapter5.tex
\chapter{Some applications of stochastic domination}
\normalsize
In this chapter we present some applications of the results proved in the last two chapters. The joint density of eigenvalues of many random matrix ensembles are known to be determinantal. Also, there are beautiful connections between many random matrix ensembles and last passage percolation. We will define a directed last passage percolation and mention a few results which connect the last passage time in a directed last passage percolation with largest eigenvalues of some random matrix ensembles. After elucidating this connection, we prove a result due to  R. Basu and S. Ganguly about the largest eigenvalues of Wishart ensemble, that is, we prove (Corollary 4.3, \cite{BG}) that $\lam^*(W_{n-k-1,n+k+1})\prec \lam^*(W_{n-k,n-k})$ for $0\le k\le n-2$. We prove an analogous result about the largest eigenvalues of Meixner ensemble, which in turn gives the stochastic domination between last passage time in directed last passage percolation with exponential weights.

\section{Random matrix ensembles and Last passage percolation}
We will introduce the directed last passage percolation (DLPP) on $\N^2.$ Consider a family of non-negative random variables $\{w(i,j):(i,j)\in \N^2\}$, called weights or passage times. And let $\Pi(m,n)$ be the set of all up-right paths $\pi$ in $\N^2$ from $(1,1)$ to $(m,n).$ Define the random variable 
\[G(m,n):=\max\limits_{\pi\in \Pi(m,n)}\summ_{(i,j)\in \pi}w(i,j).\]
This random variable $G(m,n)$ is called last passage time of $(m,n).$ The idea is that passing through a vertex $(i,j)$ takes some random amount of time which is given by the random variable $w(i,j).$ The reason for calling it a `last passage percolation' is that $G(m,n)$ is essentially the time taken to reach the point $(m,n)$ via the slowest path. The study of $G(m,n)$ naturally leads to the connections with Young tabluex, polynuclear growth model, tandem queues and totally asymmetric simple exclusion process (see \cite{DLPPP_RMT}, \cite{Johnsson}). We will not get into these details here. We are concerned only with the relation of $G(m,n)$ with various random matrix ensembles. We will be particularly concerned with  DLPP with i.i.d. exponential weight and i.i.d. geometric weight. The last passage time in these two cases `correspond' to the largest eigenvalues of \emph{Wishart ensemble} and \emph{Miexner ensemble} respectively. We define below the Wishart and Miexner ensemble and state the results which connect the last passage time to the largest eigenvalues of these ensemble.

\definition[Wishart ensemble]
If $A_{m,n}$ is a $m\times n$ matrix whose entries are independent standard complex Gaussian entries (i.e., the real and imaginary parts are i.i.d. $N(0,1/2)$), then the matrix $W_{m,n}=AA^*$ is called the complex Wishart matrix. 

Equivalently, the Wishart matrix $W_{m,n}$ also corresponds to the following measure on the space of Hermitian matrices $\mathcal{H}_m$:
$$P_{m,n}(A)dA=Z^{-1}(\det A)^{n-m}\exp(-\Tr(A))\mathbf{1}_{Y\ge 0}\; dY.$$
Where $Y\ge 0$ means that $Y$ is positive semidefinite matrix. Let $\lam^*(W_{m,n})$ denote the largest eigenvalue of $W_{m,n}$. We recall the well-known result (see section 3, equation (3.7) of \cite{DLPPP_RMT}) that the eigenvalues of $W_{m,n}$ for $m\le n$ have joint density given by
\begin{align*}
\frac{1}{Z_{m,n}}\prodd_{1\le j<k\le m}|\lam_j-\lam_k|^2 \; \; \prodd_{k=1}^m \lam_k^{n-m}e^{-\lam_k}.
\end{align*} 
We record here the following result which establishes the connection between DLPP and Wishart matrix.

\begin{proposition}[\cite{DLPPP_RMT}, Proposition 4.4]
For any $n\ge m\ge 1, t\ge 0$, the distribution for $G(m,n)$ with i.i.d. exponential weights with mean 1 is
\begin{equation}
\label{WishartDLPP}
\P(G(m,n)\le t)=Z_{m,n}^{-1}\intt_{[0,t]^m}\prodd_{1\le j<k\le m}|\lam_j-\lam_k|^2 \; \; \prodd_{k=1}^m \lam_k^{n-m}e^{-\lam_k}d\lam_k.
\end{equation}
\end{proposition}

It is obvious that $G(m,n)\le G(m,n+1)$ or $G(m,n)\le G(m+1,n)$ from the description of the $G(m,n).$ But it is not so obvious to compare the random variables $G(n,n)$ and $G(n-1,n+1).$ There is no natural way to couple these two random variables on $\N^2$. A result of the authors in \cite{BG} (see section 5) implies that there exists a coupling between $G(n,n)$ and $G(n-1,n+1)$ such that $G(n-1,n+1)\le G(n,n)$. Observe that the right hand side in the \eqref{WishartDLPP} gives the distribution of the largest eigenvalue of Wishart matrix (which is known to be determinantal). Therefore, this question can be translated in terms of the largest eigenvalues of $W_{n,n}$ and $W_{n-1,n+1}.$ Let $\lam^*(W_{p,q})$ denote the largest eigenvalue of $W_{p,q}$. The above problem is therefore equivalent to showing that $\lam^*(W_{n-1,n+1})\prec \lam^*(W_{n, n}).$ Indeed this is true, and the authors in \cite{BG}  prove (see Corollary 5.3) the following:

\begin{theorem} [R. Basu and S. Ganguly]: $\lam^*(W_{n-k-1,n+k+1})\prec \lam^*(W_{n-k,n-k})$ for $0\le k\le n-2$
\end{theorem}

Motivated by this result, we ask the same question about the last passage time in DLPP with i.i.d geometric weights. When the weights are i.i.d. geometric with parameter $q$, the distribution of last passage time $G(m,n)$ is given by the following proposition.

 \begin{proposition}[\cite{DLPPP_RMT}, Proposition 4.1]
 	For any $n\ge m\ge 1,$ the distribution for $G(m,n)$ with i.i.d. geometric weights with parameter $q$ is 
 	\begin{equation}
 	\label{MiexnerDLPP}
 	\P(G(m,n)\le t) = Z^{-1}\summ_{\stackrel{h\in \N^m:}{\max\{h_i\le t+m-1\}}}\prodd_{1\le i<j\le m}(h_i-h_j)^2\;\prodd_{i=1}^{m}{h_i+n-m\choose h_i}q^{h_i}
 	\end{equation}	
 \end{proposition}
 
And, thankfully the measure $P_{m,n}$ on $\N^m$, called \emph{Meixner ensmeble}, given by
$$Z^{-1}\summ_{\stackrel{h\in \N^m:}{\max\{h_i\le t+m-1\}}}\prodd_{1\le i<j\le m}(h_i-h_j)^2\;\prodd_{i=1}^{m}{h_i+n-m\choose h_i}q^{h_i},$$
is also determinantal. We exploit this fact and use the results proved in the last chapter along with the Lyons' result on stochastic domination to show that $G(n-1,n+1)\prec G(n,n).$ Straseen's theorem therefore gives the coupling of $G(n,n)$ and $G(n-1,n+1)$ such that $G(n-1,n+1)\le G(n,n)$. 

\section{Application of stochastic domination in random matrix ensemble}
In this section we will give three applications of the Stochastic domination results proved in the previous chapter. 

\subsection*{Stochastic domination for eigenvalues of Wishart matrix}
 
It is clear from the discussion in the previous section that this corresponds to proving the stochastic domination between the last passage time $G(n-k-1,n+k+1)$ and $G(n-k,n-k)$ of directed last passage percolation with i.i.d. exponential weights. 

Observe that if $m_1\le m_2$ and $n_1\le n_2$ then $\lam^*(W_{m_1,n_1})\prec \lam^*(W_{m_2,n_2})$. Indeed, if the two matrices are coupled in the natural way so that $W_{m_1,n_1}$ is a sub-matrix of $W_{m_2,n_2}$, then we in fact have $\lam^*(W_{m_1,n_1})\le \lam^*(W_{m_2,n_2})$.  However, this method of proof does not give the comparison between largest eigenvalues of $W_{n, n}$ and $W_{n-1, n+1}$. Instead we prove the conjecture using the determinantal structure of the eigenvalue density of $W$ and the theorem of Lyons on stochastic domination of determinantal point processes proved in the previous chapter.

\parag{Proof of the Theorem about maximum eigenvalue of Wishart matrices}
Recall that the eigenvalues of $W_{m, n}$ for $m\le n$ have joint density given by
\begin{align*}
Z_{m,n}^{-1}\intt_{[0,t]^m}\prodd_{1\le j<k\le m}|\lam_j-\lam_k|^2 \; \; \prodd_{k=1}^m \lam_k^{n-m}e^{-\lam_k}d\lam_k.
\end{align*}
To be more precise, this is the density with respect to Lebesgue measure on $\R_+^m$ of the vector of eigenvalues of $W_{m, n}$ put in uniform random order. 

 Apply Gram-Schmidt procedure to $x^{n-1},x^{n-2},\ldots ,x^{0}$ in that order in  $L^2(\R_+,e^{-x}dx)$ to get $\phi_{n-1},\ldots ,\phi_{0}$. Note that these are not Laguerre polynomials. In fact, $\phi_{n-1}(x)=\frac{1}{\sqrt{(2n-2)!}}x^{n-1}$. More generally, $\phi_k$ is a linear combination of $x^k,\ldots ,x^{n-1}$. Let $c_k$ denote the coefficient of $x^k$ in $\phi_k$. Then,
\begin{align*}
\prodd_{1\le j<k\le n-\ell}(\lam_j-\lam_k) \; \; \prodd_{k=1}^{n-\ell} \lam_k^{2\ell} &= 
\det\l[\begin{array}{cccc}
\lam_1^{\ell} & \lam_1^{\ell+1} & \ldots & \lam_1^{n-1} \\
\vdots & \vdots & \vdots & \vdots  \\
\lam_{n-\ell}^{\ell} & \lam_{n-\ell}^{\ell+1} & \ldots & \lam_{n-\ell}^{n-1}
\end{array}
\r] \\
&=\frac{1}{\prodd_{j=\ell}^{n-1} c_{j}} \det\l[\begin{array}{cccc}
\phi_{\ell}(\lam_1) & \phi_{\ell+1}(\lam_1) & \ldots & \phi_{n-1}(\lam_1) \\
\vdots & \vdots & \vdots & \vdots  \\
\phi_{\ell}(\lam_{n-\ell})  & \phi_{\ell+1}(\lam_{n-\ell}) & \ldots & \phi_{n-1}(\lam_{n-\ell}) 
\end{array}
\r].
\end{align*}
Therefore, for $0\le \ell\le n-1$, the density of eigenvalues of $W_{n-\ell,n+\ell}$ (w.r.t. the measure $(e^{-x}dx)^{\otimes n-\ell}$ on $\R_+^{n-\ell}$) is proportional to 
\begin{align*}
& \det\l(\l[\begin{array}{cccc}
\phi_{\ell}(\lam_1) & \phi_{\ell+1}(\lam_1) & \ldots & \phi_{n-1}(\lam_1) \\
\vdots & \vdots & \vdots & \vdots  \\
\phi_{\ell}(\lam_p)  & \phi_{\ell+1}(\lam_p) & \ldots & \phi_{n-1}(\lam_{n-\ell}) 
\end{array}
\r]\l[\begin{array}{cccc}
\phi_{\ell}(\lam_1) & \phi_{\ell}(\lam_2) & \ldots &  \phi_{\ell}(\lam_{n-\ell})  \\
\vdots & \vdots & \vdots & \vdots  \\
\phi_{n-1}(\lam_1)& \phi_{n-1}(\lam_2) & \ldots & \phi_{n-1}(\lam_{n-\ell}) 
\end{array}
\r]\r) \\
&=\det\l[(K_{\ell}(\lam_i,\lam_j))_{i,j\le n-\ell} \r].
\end{align*}
with $K_{\ell}(x,y)=\phi_{\ell}(x)\phi_{\ell}(y)+\ldots +\phi_{n-1}(x)\phi_{n-1}(y)$. Using the orthonormality of $\phi_j$s, a simple calculation gives the normalization constant to be $1/(n-\ell)!$. Thus, the eigenvalues of $W_{n-\ell,n+\ell}$ form a determinantal process on $(\R_+,e^{-x}dx)$ with kernel $K_{\ell}$.
Now, let $\ell \le n-2$ and apply Theorem~\ref{thm:determinantaldomination} to the eigenvalues of $W_{n-\ell,n+\ell}$ and $W_{n-\ell-1,n+\ell+1}$ to see that the latter is stochastically dominated by the former. 

Now fix $t>0$ and consider the set $\AA$ of all measures $\theta$ on $\R_+$ such that $\theta([t,\infty))>0$. This is an increasing set of measures. If $\XX$ is the counting measure of eigenvalues of $W_{n-\ell,n+\ell}$ and $\YY$ is the counting measure of eigenvalues of $W_{n-\ell-1,n+\ell+1}$, then it follows by the definition of stochastic domination that $\P\{\XX\in \AA\}\ge \P\{\YY\in \AA\}$. But $\P\{\XX\in \AA\}$ is the same as $\P\{\lam^*(W_{n-\ell,n+\ell})\ge t\}$ and similarly for $\YY$. Thus, the desired stochastic domination of largest eigenvalues follows.

\subsection*{Stochastic domination for eigenvalues  of Jacobi ensemble}
Jacobi ensemble is a family of p.d.fs given by 
\begin{equation}
\label{jacobi}
C_{n,a,b,\beta}^{-1}\prod_{j=1}^{n}(1-x_j)^{a\beta/2}(1+x_j)^{b\beta/2}\prod_{1\le i<j\le n} |x_i-x_j|^{\beta}, \hspace{3mm} x_j\in [-1,1],
\end{equation}
where $C_{n,a,b,\beta}$ is the suitable normalizing constant. This family of joint-densities does arise naturally as the joint-density of eigenvalues of some random matrices at least when $\beta=1,2,4.$ For details we refer the reader to Chapter 3 of  \cite{Forrester}. We are only interested in the case $\beta=2.$ In this case, the above density arises as the joint density of eigenvalues as follows.

\begin{proposition}[proposition 3.6.1.,\cite{Forrester}]
The eigenvalues of $n\times n$ matrix 
$$J_{n_1,n_2,n}=\frac{AA^*}{AA^*+BB^*}$$ where $A, B$ are matrices of size $n\times n_1$ and $n\times n_2$ respectively with i.i.d. standard complex Gaussian entries has joint density given by
$$C_{n,n_1,n_2}^{-1}\prod_{j=1}^{n}(x_j)^{n_1-n}(1-x_j)^{n_2-n}\prod_{1\le i<j\le n} |x_i-x_j|^{2}, \hspace{3mm} x_j\in [0,1].$$
\end{proposition}

\begin{proposition}[Stochastic domination for $\beta=2$ Jacobi ensemble] 
$$\lam^*(J_{n_1+1,n_2-1,n-1})\prec \lam^*(J_{n_1,n_2,n}).$$
\end{proposition}
\begin{proof}
The proof is verbatim same as in the case of Wishart ensemble, but we give the proof for completeness. Recall from the previous proposition that the joint density of the eigenvalues of $J_{n_1,n_2,m}$ is proportional to 
$$\prod_{j=1}^{n}(x_j)^{n_1-n}(1-x_j)^{n_2-n}\prod_{1\le i<j\le n} |x_i-x_j|^{\beta}, \hspace{3mm} x_j\in [0,1].$$

Let $\phi_{n-1},\ldots ,\phi_{0}$ be orthonormal vectors in  $L^2([0,1],x^{n_1-n}x^{n-2-n}dx)$ obtained by Gram-Schmidt procedure applied to $x^{n-1},x^{n-2},\ldots ,x^{0}$ in that order. Let $c_k$ denote the coefficient of $x^k$ in $\phi_k$. Then,
\begin{align*}
\prodd_{1\le j<k\le n}(\lam_j-\lam_k) \; \;  &= 
\det\l[\begin{array}{cccc}
\lam_1^{0} & \lam_1^{1} & \ldots & \lam_n^{n-1} \\
\vdots & \vdots & \vdots & \vdots  \\
\lam_{n}^{0} & \lam_{n}^{1} & \ldots & \lam_{n}^{n-1}
\end{array}
\r] \\
&=\frac{1}{\prodd_{j=1}^{n-1} c_{j}} \det\l[\begin{array}{cccc}
\phi_{0}(\lam_1) & \phi_{1}(\lam_1) & \ldots & \phi_{n-1}(\lam_1) \\
\vdots & \vdots & \vdots & \vdots  \\
\phi_{0}(\lam_{n-1})  & \phi_{1}(\lam_{n-1}) & \ldots & \phi_{n-1}(\lam_{n-1}) 
	\end{array}
	\r]
\end{align*}
Therefore, w.r.t. the measure $(x^{n_1-n}x^{n_2-n}dx)^{\otimes n}$ on $[0,1]$,  the density of eigenvalues of $J_{n_1,n_2,m}$ is proportional to $\det(K_{n}(\lam_i,\lam_j))_{1\le i,j\le n}$ where $K_{n}(x,y)=\sum\limits_{j=0}^{n-1}\phi_{j}(x)\phi_{j}(y)$. Therefore, the eigenvalues of $J_{n_1,n_2,n}$ form a determinantal process with the kernel $K_n(x,y)$ w.r.t. the measure $(x^{n_1-n}(1-x)^{n_2-n}dx).$ 
Exactly similar computation shows that the eigenvalues of $J_{n_1+1, n_2-1, n-1}$ form a determinantal process with the kernel $\sum_{j=1}^{n-1}\phi_j(x)\phi_j(y)$ with respect to the measure $(x^{n_1-n}(1-x)^{n_2-n}).$
 Invoking the theorem \ref{thm:determinantaldomination} gives and repeating exactly the same argument as in the last paragraph of the previous section, we obtain the desired result.  
\end{proof}

Of course a similar strategy also gives us that $\lam^*(J_{n_1-1,n_2+1,n-1})\prec \lam^*(J_{n_1,n_2,n}).$ We wish to point out here is that the general scheme here is to first show that eigenvalues of some ensemble is determinantal (which in both of the above cases is fairly well known). Then we go on to compute the kernels of these determinantal processes. The key step is to observe that it is possible to subsume some part of the measure into the kernel so that both kernels are expressed w.r.t. a common reference measure. After that its just a matter of checking the condition in the Lyons theorem, and invoke the Lyons theorem. In the next section we deal with Meixner ensemble and the reason why Lyons theorem is not directly applicable in that case is precisely that we are not able to represent the two kernels with respect to a common reference measure.

\subsection*{Stochastic domination for eigenvalues of Meixner ensemble}
Recall that Meixner ensemble $M(m,n), m\ge n$ is given by the following probability measure on $\N^m$ (for $0<q<1$,) 
$$Z_{m,n}^{-1}\prodd_{1\le i<j\le m}(h_i-h_j)^2\;\prodd_{i=1}^{m}{h_i+n-m\choose h_i}q^{h_i}.$$

The goal here is to compare the law of rightmost particles of $M(n,n)$ and $M(n-1,n+1).$ The joint density of the particles of $M(n,n)$ is give by 
$$Z_{n,n}^{-1}\prodd_{i=1}^{n}(h_i-h_j)^2\prodd_{i=1}^{n}q^{h_i}.$$
Arguing exactly as in the case of Wishart matrix, it can be shown that this is determinantal with the kernel $K_n(x,y)=\summ_{i=0}^{n-1}\phi_{i}(x)\phi_{i}(y)$ where $\phi_{n-1}, \ldots, \phi_0$ are the vectors obtained by orthonormalizing $x^{n-1}, \ldots, x, 1$ w.r.t. the probability measure with probability mass function proportional to $\mu(x)=q^x.$ Let us proceed as we did earlier and write the joint density of particles of $M(n+1,n-1)$ which is given by 
\[
Z_{n-1,n+1}^{-1}\prodd_{i=1}^{n-1}(h_i-h_j)^2\;\prodd_{i=1}^{n-1}(h_i+2)(h_i+1)q^{h_i}.
\]
Observe that we still have a determinantal process with kernel $K'(x,y)=\summ_{i=0}^{n-2}\psi_i(x)\psi_i(y)$ were $\psi_i$ are obtained by orthonormalizing $x^{n-2},\ldots, x, 1$ but with respect to the probability measure on $\N$ with probability mass function proportional to $(x+2)(x+1)q^x.$ In the earlier examples, we could get the kernel $K'$ by by orthonormalizing $x^{n-2},\ldots, x, 1$ w.r.t measure with p.m.f proportional to $x^2(x)d\mu(x).$ In that case it was possible to subsume this polynomial term $x^2$ into the determinant term, and thus express this kernel w.r.t. the original measure $\mu(x)$ so that $K'(x,y)=\summ_{i=1}^{n-1}\phi_{i}(x)\phi_{i}(y),$ and therefore theorem \ref{thm:determinantaldomination} could be used to compare the two processes. But in this case, theorem \ref{thm:determinantaldomination} is not  directly applicable. Nonetheless, it is true that $\lambda^*(M_{n+1,n-1})\prec \lambda^*(M_{n,n}),$ where $\lambda^*(M_{m, n})$ is the rightmost particle of $M_{m, n}.$ This is the content of the next proposition.

\begin{proposition}
$\lambda^*(M_{n+1,n-1})\prec \lambda^*(M_{n,n}).$	
\label{meixnerdomination}
\end{proposition}
The proof of the above proposition follows from the two claims which we will prove below. 
\begin{claim}
Let $\P_1$ be the joint law of $n$ particles on $\N$ given by  
\begin{align*}
Z_{n,n}^{-1}\prodd_{i=1}^{n}(h_i-h_j)^2\prodd_{i=1}^{n}q^{h_i}.
\end{align*}
Let $\P_2$ be the joint law of $n-1$ particles on $N$ given by 
$$Z^{-1}\prodd_{i=1}^{n-1}(h_i-h_j)^2\prodd_{i=1}^{n-1}h_i^2q^{h_i}.$$
Let $\X_1$ and $\X_2$ be the point processes obtained by the considering the unlabeled particles from $N$ chosen according to $\P_1$ and $\P_2$ respectively. Then, $\X_2\prec \X_1$. 
\end{claim}
The proof in this case is verbatim same as in the case of Wishart matrix and follows from the theorem \ref{thm:determinantaldomination}.
\begin{claim}
	Let $\P_2$ be the joint law of $n-1$ particles on $N$ given by 
	$$Z^{-1}\prodd_{i=1}^{n-1}(h_i-h_j)^2\prodd_{i=1}^{n-1}h_i^2q^{h_i}.$$
	Let $\P_3$ be the joint law of $n-1$ particles on $N$ given by 
	$$Z_{n+1,n-1}^{-1}\prodd_{i=1}^{n-1}(h_i-h_j)^2\prodd_{i=1}^{n-1}(h_i+2)(h_i+1)q^{h_i}.$$
	Let $X_2$ and $X_3$ be the point processes obtained by the considering the unlabeled particles from $N$ chosen according to $\P_2$ and $\P_3$ respectively. Then, $\X_3\prec \X_2$.
\end{claim}
\begin{proof}
Observe that $\frac{x^2}{(x+2)(x+1)}$ is an increasing function on $\N$. And, therefore the claim follows from the Theorem \ref{thm:Stochasticdominationnew}.
\end{proof}

\begin{proof}[Proof of Proposition \ref{meixnerdomination}]

Observe that $\P_3$ is the joint distribution of eigenvalue of $M(n+1,n-1)$ and $\P_1$ is the joint distribution of eigenvalues of $M(n,n).$ It follows from the last two claims that $X_3\prec X_1$ where $\X_3 (\text{or }\X_1)$ is the counting measure of the eigenvalues of $M_{n-1,n+1}(\text{ or }M_{n,n}).$ After this repeating exactly the same argument as in the Wishart's case give us the desired result.
\end{proof}
\newpage
\null